\documentclass[12pt]{amsart}
\usepackage{bbm}
\usepackage[pdftex]{graphicx}
\usepackage{amscd, amssymb, dsfont, upgreek, mathrsfs}
\input{xy}
\xyoption{all}

\newtheorem{Thm}{Theorem}
\newtheorem{Conj}[Thm]{Conjecture}
\newtheorem{Prop}[Thm]{Proposition}

\newtheorem{Def/Thm}[Thm]{Definition/Theorem}

\newtheorem{Lemma}[Thm]{Lemma}

\theoremstyle{remark}
\newtheorem{Rmk}[Thm]{Remark}

\newcommand{\ti }{\times}

\newcommand{\lann}{\langle\langle}
\newcommand{\rann}{\rangle\rangle}

\newcommand{\PP }{{\mathbb P}}

\newcommand{\QQ }{{\mathbb Q}}
\newcommand{\CC }{{\mathbb C}}
\newcommand{\ZZ }{{\mathbb Z}}

\newcommand{\vir}{\mathrm{vir}}

\newcommand{\T}{{\mathsf{T}}}

\newcommand{\lan}{\langle}
\newcommand{\ran}{\rangle}

\newcommand{\com}{{\mathbb{C}}}

\begin{document}

\title[Gromov-Witten invariants of Calabi-Yau fibrations]
{Gromov-Witten invariants of Calabi-Yau fibrations}

\author{Hyenho Lho}
\address{Department of Mathematics, ETH Z\"urich}
\email {hyenho.lho@math.ethz.ch}
\date{April 2019}.

\begin{abstract} 
We study the quasimap invariants of elliptic and K3 fibrations. Oberdieck and Pixton conjectured that the Gromov-Witten potentials of elliptic fibrations are quasi-modular forms. Analogously, we propose similar conjecture for the quasimap potentials of elliptic fibrations. We also conjecture some finite generation properties of quasimap potentials of K3 fibrations. Via wall-crossing conjecture, this will imply some quasi-modularity of the Gromov-Witten potentials of K3 fibrations. We provide some evidences for our conjectures through several examples. The method here can be further generalized to arbitrary n-dimensional Calabi-Yau fibrations.
\end{abstract}

\maketitle

\setcounter{tocdepth}{1} 
\tableofcontents

\setcounter{section}{-1}

\section{Introduction}
\subsection{Elliptic fibrations}

Consider non-singular algebraic varieties $X$, $B$ and elliptic fibration 
$$\pi : X \rightarrow B\,,$$
i.e. a flat morphism with fibers connected curves of arithmetic genus 1. Assume that $\pi$ has integral fibers and has a section

$$\iota:B\rightarrow X\,.$$
Let $N_{\iota}$ be the normal bundle of $\iota$. Fix a curve class $\gamma \in H_2(B,\ZZ)$ and let

$$\mathcal{F}^{\mathsf{GW}}_{g,\gamma}(Q)=\sum_{ \pi_*\beta=\gamma}Q^\beta\int_{[\overline{M}_{g}(X,\beta)]^{\text{vir}}}1\, \in \CC[[Q]]\,,$$
be the Gromov-Witten series associated to $\gamma$.

The ring of quasimodular forms is the free polynomial algebra
$$\mathsf{QMod}=\QQ[E_2, E_4, E_6]\,,$$
where $E_k$ are the weight $k$ Eisenstein series
$$E_k(Q)=1-\frac{2k}{B_k}\sum_{n=1}^{\infty}\frac{n^{k-1} Q^n}{1-Q^n}\,,$$
and $B_k$ are the Bernoulli numbers.
Based on the results for Gromov-Witten invariants of elliptic curves, Oberdieck and Pixton made the following conjecture (\cite{ObPix}).

\begin{Conj}
 For $\gamma \in H_2(B,\ZZ)$, $\mathcal{F}^{\mathsf{GW}}_{g,\gamma}$ is a quasimodular form:
 $$\mathcal{F}^{\mathsf{GW}}_{g,\gamma} \in \frac{1}{\Delta(Q)^m}\mathsf{QMod}\,,$$
where $\Delta=Q\prod_{n=1}^\infty(1-Q^n)^{24}$ is the modular discriminant and $m=-\frac{1}{2}c_1(N_\iota)\cdot \gamma$.
\end{Conj}

Motivated by the above conjecture, we study the quasimap theory of $X$. The quasimap invariants are introduced in \cite{CK, CKM} for the study of mirror symmetry. The relationship of quasimap invariants to Gromov-Witten invariants, which is called wall-crossing conjecture, were studied in \cite{CKg0, CKw, CKg, CJR}. Based on the wall-crossing conjecture, a fundamental relationship between the quasimap invariants and the B-model for all local and complete intersection Calabi-Yau 3-folds $X$ is natural to propose : {\em the quasimap invariants of $X$ exactly equal the B-model invariants of the mirror Y.} Therefore, via quasimap theory, we study the string theoretic B-model theory directly.  See introduction of \cite{LP}, for more details.

For a curve class $\gamma \in H_2(B,\ZZ)$, let

\begin{align*}
    \mathcal{F}^{\mathsf{SQ}}_{g,\gamma}(q)=\sum_{ \pi_*\beta=\gamma}q^\beta\int_{[\overline{Q}_{g}(X,\beta)]^{\text{vir}}}1\, \in \CC[[q]]\,,
\end{align*}
be the quasimap series associated to $\gamma$.
The following series in $q$ will play basic role.
\begin{align*}
    L(q)&=(1-27q)^{-\frac{1}{3}}=1+9q+162q^2+\dots\,,\\
    I_1^{E}(q)&=3\sum_{n=1}^{\infty}\frac{(3d-1)!}{(d!)^3}q^d\,,\\
    B_1'(q)&=q \frac{\partial}{\partial q}I^{E}_1\,,\\
    X(q)&=\frac{q \frac{\partial}{\partial q} B_1'}{1+B_1'}\,.
\end{align*}
We define the ring of quasimap elliptic fibrations 
\begin{align*}
    \mathsf{QEF}:=\CC[L^{\pm 3},B_1',X]
\end{align*}
Define differential operator

$$\mathsf{D}=q \frac{d}{d q}\,.$$
The following equations were obtained in \cite{LP}.
\begin{align*}
    &\mathsf{D}L=\frac{L}{3}(L^3-1)\,,\\
    &X^2-(L^3-1)X+\mathsf{D}X-\frac{2}{9}(L^3-1)=0\,.
\end{align*}
Via the above relations, the ring $\mathsf{QEF}$ is closed under the action of $\mathsf{D}$.
Motivated by Conjecture 1, we conjecture the following.

\begin{Conj}
 For $\gamma \in H_2(B,\ZZ)$, we have
 $$\mathcal{F}^{\mathsf{SQ}}_{g,\gamma} \in \mathsf{QEF}\,.$$
\end{Conj}
The following Lemma relate the Conjecture 1 and 2.
\begin{Lemma}[\cite{ASYZ}]
  Via the change of variable by the mirror map of elliptic curve
  
  $$
      Q=q \,  \text{Exp}(I^{E}_1(q))\,,
  $$
we have
  \begin{align*}
      E_2&=\frac{(1+B_1')^2}{L^3}(12 X+4-3L^3)\,,\\
      E_4&=\frac{(1+B_1')^4}{L^6}(-8 L^3+9L^6)\,,\\
      E_6&=\frac{(1+B_1')^6}{L^9}(-8L^3+36 L^6-27 L^9)\,.
  \end{align*}
\end{Lemma}

\

\subsection{K3 fibrations.} We can also consider Calabi-Yau 3-fold X , curve B and K3 fibration
$$\pi : X\rightarrow B\,,$$
where generic fibers of $\pi$ are smooth K3 surfaces. For a fixed curve class $\gamma \in H_2(B,\ZZ)$, Let

\begin{align*}
    \mathcal{F}^{\mathsf{GW}}_{g,\gamma}(Q)=\sum_{\pi_*\beta=\gamma}Q^\beta\int_{[\overline{M}_{g}(X,\beta)]^{\text{vir}}}1\,,\\
    \mathcal{F}^{\mathsf{SQ}}_{g,\gamma}(q)=\sum_{\pi_*\beta=\gamma}q^\beta\int_{[\overline{Q}_{g}(X,\beta)]^{\text{vir}}}1\,,
\end{align*}
be the Gromov-Witten and quasimap series associated to $\gamma$, respectively. We  define the following series in $q$.

\begin{align}\label{KG}
        \nonumber L(q)&=(1-4^4q)^{-\frac{1}{4}}=1+9q+162q^2+\dots\,,\\
        I^{K3}_1&=\frac{4\sum_{n=1}^{\infty}\frac{(4d)!}{(d!)^4}(\sum_{r=d+1}^{4d}\frac{1}{r})q^d}{\sum_{n=0}^{\infty}\frac{(4d)!}{(d!)^4}q^d}\,,\\
    \nonumber A_1'(q)&=q \frac{\partial}{\partial q}I^{K3}_1\,,\\
   \nonumber X(q)&=\frac{q \frac{\partial}{\partial q} A_1'}{1+A_1'}\,.
\end{align}
We define the ring of quasimap K3 fibrations

$$\mathsf{QKF}:=\CC[L^{\pm 4},A_1',X]\,.$$
Via the following relations which are obtained in \cite{Lho},
\begin{align*}
    &\mathsf{D}L=\frac{L}{4}(L^4-1)\,,\\
    &X^2-2 \mathsf{D}X+\frac{1}{16}(12L^8-11L^4-1)=0\,.
\end{align*}
the ring $\mathsf{QEF}$ is closed under the action of $\mathsf{D}$.

\begin{Conj}
 For $\gamma \in H_2(B,\ZZ)$, We have
 $$\mathcal{F}^{\mathsf{SQ}}_{g,\gamma}\in \mathsf{QKF}\,.$$
\end{Conj}

It is an interesting question to find the modular interpretation of the generators in $\mathsf{QKF}$ after change of variable by mirror map  
$$Q=q Exp(I^{K3}_1(q))\,.$$ 
This will imply the modularity of the Gromov-Witten series of K3-fibrations.

\

\subsection{Twisted theories on $\PP^{n_1} \times \PP^{n_2}$}
Twisted theories associated to $\PP^m$ was introduced in \cite{LP}. We similarly define twisted theories associated to $\PP^{n_1} \times \PP^{n_2}$ and study several cases.

Let $\mathsf{T}_{n_1+1,n_2+1}$ be the algebraic torus
$$(\CC^*)^{n_1+1}\times (\CC^*)^{n_2+1}$$
and let each component $(\CC^*)^{n_i+1}$ of  $\mathsf{T}_{n_1+1,n_2+1}$
act with the standard linearization on $\PP^{n_i}$ with weights $\lambda_{i,0},\dots,\lambda_{i,n_i}$ on the vector space $$H^0(\PP^{n_i},\mathcal{O}_{\PP^{n_i}}(1))\,.$$ 

Let $\overline{M}_{g}(\PP^{n_1}\times\PP^{n_2},(d_1,d_2))$ be the moduli space of stable maps to $\PP^{n_1}\times\PP^{n_2}$ with the canonical $\mathsf{T}_{n_1+1,n_2+1}$-action, and let
\begin{align*}
\mathsf{C}\rightarrow \overline{M}_g(\PP^{n_1}\times\PP^{n_2},(d_1,d_2))\,,\,\,f:\mathsf{C}\rightarrow\PP^{n_1}\ti\PP^{n_2}\,,\,\,
\mathsf{S}_{i} = f^*\mathcal{O}_{\PP^{n_i}(-1)}
\end{align*}
be the standard universal structures. Let
$$\mathsf{a}=(a_{11},\dots,a_{1r};a_{21},\dots,a_{2r})\,,\,\,\,\mathsf{b}=(b_{11},\dots,b_{1s};b_{21},\dots,a_{2s})$$
be vectors of positive integers satisfying the Calabi-Yau conditions
$$\sum_{i=1}^{r}a_{ki}-\sum_{j=1}^s b_{kj}=n_k+1\,\,\,\,\text{for}\,\,k=1,2\,.$$

The Gromov-Witten invariants of the $(\mathsf{a},\mathsf{b})$-twisted geometry of $\PP^{n_1}\times\PP^{n_2}$ are defined by the equivariant integrals
$$\widetilde{N}_{g,(d_1,d_2)}^{\mathsf{GW}}=\int_{[\overline{M}_g(\PP^{n_1}\times\PP^{n_2},(d_1,d_2))]}\prod_{i=1}^r e(R\pi_*\mathsf{S}^{-a_{1i}}\otimes\mathsf{S}^{-a_{2i}})\prod_{j=1}^s e(-R\pi_*\mathsf{S}^{b_{1j}}\otimes\mathsf{S}^{b_{2j}})\,.$$
The above integral define a rational function in equivariant variables $\lambda_{ki}$
$$\widetilde{N}_{g,(d_1,d_2)}^{\mathsf{GW}}\in \CC(\lambda_{1,0},\dots,\lambda_{1,n_1},\lambda_{2,0},\dots,\lambda_{2,n_2})\,.$$

Using the moduli space $\overline{Q}_g(\PP^{n_1}\times\PP^{n_2},(d_1,d_2))$ of stable quasimaps with the standard structures,

\begin{align*}
&\mathsf{C}\rightarrow \overline{Q}_g(\PP^{n_1}\times\PP^{n_2},(d_1,d_2))\,,\,\,f:\mathsf{C}\rightarrow[\CC^{n_1+1}\times\CC^{n_2+1}/\CC^*\times\CC^*]\,, \\
&\mathsf{S}_{i} = f^*\mathcal{O}_{[\CC^{n_i+1}/\CC^*]}(-1)\,,
\end{align*}
the quasimap invariants of the $(\mathsf{a},\mathsf{b})$-twisted geometry of $\PP^{n_1+1}\times\PP^{n_2+1}$ are defined by the equivariant integrals

$$\widetilde{N}_{g,(d_1,d_2)}^{\mathsf{SQ}}=\int_{[\overline{M}_g(\PP^{n_1}\times\PP^{n_2},(d_1,d_2))]}\prod_{i=1}^r e(R\pi_*\mathsf{S}^{-a_{1i}}\otimes\mathsf{S}^{-a_{2i}})\prod_{j=1}^s e(-R\pi_*\mathsf{S}^{b_{1i}}\otimes\mathsf{S}^{b_{2i}})\,.$$

\begin{Rmk}
By Quantum Lefchetz theorem \cite{KKP}, for genus 0 and 1, the $(\mathsf{a},\mathsf{0})$ - twisted quasimap theories of $\PP^{n_1}\times\PP^{n_2}$ recover the quasimap theories of complete intersections of degree $(\mathsf{a})$ in $\PP^{n_1}\times\PP^{n_2}$. For higher genus, more techniques are required to obtain the theories of complete intersections from twisted theories. For quintic threefolds, this was studied using NMSP theory \cite{CGL} or log GLSM theory \cite{GJR}. Using these approaches we expect to be able to study the quasimap theories of Calabi-Yau fibrations for higher genus. We will return to these problems in the future.  
\end{Rmk}

\subsection{Main results}
\subsubsection{Local $\PP^1\times\PP^1$ theory.} Let 
$$\mathcal{F}^{\mathsf{SQ}}_g(q_1,q_2)=\sum_{d_1,d_2\ge 0}N^{\mathsf{SQ}}_{g,(d_1,d_2)}q_1^{d_1}q_2^{d_2}\in\CC[[q_1,q_2]]\,$$
be the genus g quasimap series of local $\PP^1 \times \PP^1$.
In order to state the theorem, we define a series in $q_1$.
$$X(q_1)=(1-8q_1)^{-1/2}=1+4q_1+24q_1^2+160q_1^3+\dots$$

\begin{Thm}\label{MT1}
 For the quasimap invariants of $K_{\PP^1\times\PP^1}$, we have
\begin{itemize}
    \item[(i)] $
     q_1\frac{\partial}{\partial q_1}\mathcal{F}^{\mathsf{SQ}}_1\in\CC[[q_2]][X^2,(1+X^2)^{-1}]\,,$
    \item[(ii)]  $
     \mathcal{F}^{\mathsf{SQ}}_g\in\CC[[q_2]][X^2,(1+X^2)^{-1}]\,\,\,\,\,\text{for}\,\,g\ge2\,.$
 \end{itemize}
In other words, the coefficients of $q_2^k$ in $q_1\frac{\partial}{\partial q_1}\mathcal{F}^{\mathsf{SQ}}_{1},\,\mathcal{F}^{\mathsf{SQ}}_{g\ge2}$ are elements in the ring $$\CC[X^2,(1+X^2)^{-1}]\,.$$
\end{Thm}

\subsubsection{Elliptic fibration : K3 surface}
Let 

\begin{align*}
    \mathcal{F}_{g}^{\mathsf{SQ}}(q_1,q_2)=\sum_{d_1,d_2 \ge 0} N^{\mathsf{SQ}}_{g,(d_1,d_2)} q_1^{d_1}q_2^{d_2}\in \CC[[q_1,q_2]]
\end{align*}
be the
$((3;2),(0;0))$-twisted genus $g$ quasimap series of $\PP^2 \times \PP^1$.

Recall the series which generate the ring $\mathsf{QEF}$,
\begin{align}\label{EG}
    \nonumber L(q_1)&=(1-27q_1)^{-\frac{1}{3}}=1+9q_1+162q_1^2+\dots\,,\\
    B_1'(q_1)&=q_1 \frac{\partial}{\partial q_1}I^{E}_1\,,\\
    \nonumber X(q_1)&=\frac{q_1 \frac{\partial}{\partial q_1} B_1'}{1+B_1'}\,.
\end{align}

\begin{Thm}\label{MT2}
  For the $((3;2),(0;0))$-twisted genus g quasimap series of $\PP^2 \times \PP^1$, we have
  \begin{itemize}
      \item[(i)] $q_1\frac{\partial}{\partial q_1}\mathcal{F}_{1}^{\mathsf{SQ}}(q_1,q_2)\in\CC[[q_2]](L^3,B_1',X)\,,$
      \item[(ii)] $\mathcal{F}_{g}^{\mathsf{SQ}}(q_1,q_2)\in\CC[[q_2]](L^3,B_1',X)\,\,\text{for}\,\,g\ge2\,.$
  \end{itemize}
\end{Thm}
From the low degree calculations, we make the following conjecture.

\begin{Conj}
  For the $((3;2),(0;0))$-twisted genus g quasimap series of $\PP^2 \times \PP^1$, we have
 \begin{itemize}
      \item[(i)] $q_1\frac{\partial}{\partial q_1}\mathcal{F}_{1}^{\mathsf{SQ}}(q_1,q_2)\in\CC[[q_2]][L^3,B_1',X]\,,$
      \item[(ii)] $\mathcal{F}_{g}^{\mathsf{SQ}}(q_1,q_2)\in\CC[[q_2]][L^3,B_1',X]\,\,\text{for}\,\,g\ge2\,.$
  \end{itemize}
\end{Conj}
For example, in genus 1 we have

$$q_1\frac{\partial}{\partial q_1}\mathcal{F}_1^{\mathsf{SQ}}(q_1,0)=-X\,.$$

\

\subsubsection{Elliptic fibration : CY 3-fold.} Let
\begin{align*}
    \mathcal{F}_{g}^{\mathsf{SQ}}(q_1,q_2)=\sum_{d_1,d_2 \ge 0} N^{\mathsf{SQ}}_{g,(d_1,d_2)} q_1^{d_1}q_2^{d_2}\in \CC[[q_1,q_2]]
\end{align*}

be the $((3;3),(0;0))$-twisted genus $g$ quasimap series of $\PP^2\times\PP^2$. We use the same generators \eqref{EG} in the following theorem.

\begin{Thm}\label{MT3}
For the $((3;3),(0;0))$-twisted genus g quasimap series of $\PP^2 \times \PP^2$, we have
\begin{itemize}
    \item[(i)] $q_1\frac{\partial}{\partial  q_1}\mathcal{F}_{1}^{\mathsf{SQ}}(q_1,q_2)\in\CC[[q_2]](L^3,B_1',X)\,,$
    \item[(ii)] $
      \mathcal{F}_{g}^{\mathsf{SQ}}(q_1,q_2)\in\CC[[q_2]](L^3,B_1',X)\,\,\,\text{for}\,\,g\ge2\,.$
\end{itemize}
\end{Thm}

From the genus one calculation
$$q_1\frac{\partial}{\partial q_1}\mathcal{F}_1^{\mathsf{SQ}}(q_1,0)=-\frac{1}{4}(L^3-1)-\frac{3}{2}X\,,$$
we make the following conjecture.

\begin{Conj}
  For the $((3;3),(0;0))$-twisted genus g quasimap series of $\PP^2 \times \PP^2$, we have
\begin{itemize}
    \item[(i)] $q_1\frac{\partial}{\partial  q_1}\mathcal{F}_{1}^{\mathsf{SQ}}(q_1,q_2)\in\CC[[q_2]][L^3,B_1',X]\,,$
    \item[(ii)] $
      \mathcal{F}_{g}^{\mathsf{SQ}}(q_1,q_2)\in\CC[[q_2]][L^3,B_1',X]\,\,\,\text{for}\,\,g\ge2\,.$
\end{itemize}
\end{Conj}
\

\subsubsection{K3 fibration}
 Let
\begin{align*}
    \mathcal{F}_{g}^{\mathsf{SQ}}(q_1,q_2)=\sum_{d_1,d_2 \ge 0} N^{\mathsf{SQ}}_{g,(d_1,d_2)} q_1^{d_1}q_2^{d_2}\in \CC[[q_1,q_2]]
\end{align*}
be the $((4;2),(0;0))$-twisted genus $g$ quasimap series of $\PP^3\times\PP^1$. In order to state the theorem, we need the generators \eqref{KG} of the ring $\mathsf{QKF}$ and the extra generators $E_1', B_2'$ whose definition will appear in  Section \ref{K3Gen}.

\begin{Thm}\label{MT4}
For the $((4;2),(0;0))$-twisted genus g quasimap series of $\PP^3 \times \PP^1$, we have
  \begin{itemize}
    \item[(i)] $q_1\frac{\partial}{\partial q_1}
      \mathcal{F}_{1}^{\mathsf{SQ}}(q_1,q_2)\in\CC[[q_2]](L^4,A_1',X,E_1',B_2')\,,$
  \item[(ii)] $
      \mathcal{F}_{g}^{\mathsf{SQ}}(q_1,q_2)\in\CC[[q_2]](L^4,A_1',X,E_1',B_2')\,\,\,\text{for}\,\,g\ge2\,.$
  \end{itemize}
\end{Thm}

Using the argument in the proof of the theorem in Section \ref{MS4}, one can show that the coefficient of $q_2^d$ in $\mathcal{F}_{g}^{\mathsf{SQ}}$ do not have extra series $E_1'$ and $B_2'$ for fixed $d$.

For example in the genus 1 case, we have
$$q_1\frac{\partial}{\partial q_1}\mathcal{F}_1^{\mathsf{SQ}}(q_1,0)=\frac{13}{12}(1-L^4)+2X\,.$$
From this observation, we conjecture the stronger result.

\begin{Conj}
  For the $((4;2),(0;0))$-twisted genus g quasimap series of $\PP^3 \times \PP^1$, we have
 \begin{itemize}
    \item[(i)] $q_1\frac{\partial}{\partial q_1}
      \mathcal{F}_{1}^{\mathsf{SQ}}(q_1,q_2)\in\CC[[q_2]][L^4,A_1',X]\,,$
  \item[(ii)] $
      \mathcal{F}_{g}^{\mathsf{SQ}}(q_1,q_2)\in\CC[[q_2]][L^4,A_1',X]\,\,\,\text{for}\,\,g\ge2\,.$
  \end{itemize}
\end{Conj}

\subsection{Plan of the paper.}
 We will prove Theorem \ref{MT1}, \ref{MT2}, \ref{MT3} and \ref{MT4} in Section \ref{MS1}, \ref{MS2}, \ref{MS3} and \ref{MS4}, respectively.

\subsection{Acknowledgments} 
I am very grateful to  I.~Ciocan-Fontanine, S. ~Guo, H.~Iritani, F. ~Janda, B.~Kim, A.~Klemm, S. ~Lee, M. C.-C. Liu, G. ~Oberdieck, R. ~Pandharipande, Y. ~Ruan, E.~Scheidegger and J. ~Shen 
for discussions over the years 
about the moduli space of quasimaps and the invariants of Calabi-Yau geometries. I was supported by the grant ERC-2012-AdG-320368-MCSK and
ERC-2017-AdG-786580-MACI.

This project has received funding from the European Research Council (ERC) under the European Union’s
Horizon 2020 research and innovation program (grant agreement No 786580).

\
\

\section{Local $\PP^1\times\PP^1$}\label{MS1}

\subsection{Overview.} Let $K_{\PP^1\times\PP^1}$ be the total space of the canonical bundle over $\PP^1\times\PP^1$. The $((0;0),(-2;-2))$-twisted theory on $\PP^1\times\PP^1$ recovers the standard theory of local $\PP^1\times\PP^1$. Since the quasimap invariants are independent of $\lambda_{i,k}$, we are free to use the specialization 
\begin{align}\label{P1P1sp}
\lambda_{1,1}=1\,,\,\lambda_{1,2}=-1\,,\,\lambda_{2,1}=\sqrt{-1}\,,\,\lambda_{2,2}=-\sqrt{-1}\,.
\end{align}
The specialization \eqref{P1P1sp} will be imposed for our entire study of $K_{\PP^1\times\PP^1}$. For simplicity we use following notations;
\begin{align*}
    \alpha_0=\lambda_{1,1}\,,\,\alpha_1=\lambda_{1,2}\,,\,\beta_0=\lambda_{2,1}\,,\,\beta_1=\lambda_{2,2}\,.
\end{align*}

\subsection{Geneartors.}\label{Gen}
From the small $I$-function associated to $K_{\PP^1\times\PP^1}$,
\begin{align*}
    I(q_1,q_2)=\sum_{d_1,d_2\ge 0}q_1^{d_1}q_2^{d_2}
    \frac{\prod_{k=0}^{2d_1+2d_2-1}(-2H_1-2H_2-kz)}{\prod_{i=0}^1\prod_{k=1}^{d_1}(H_1-\alpha_i+kz)\prod_{j=0}^1\prod_{k=1}^{d_2}(H_1-\beta_j+kz)}\,,
\end{align*}
define the series $I_{ij}$ by the equation

\begin{align*}
    I=&1+\Big(I_{11}H_1+I_{12}H_2\Big)\frac{1}{z}+\Big(I_{21}H_1^2+I_{22}H_1 H_2+I_{23}H_2^2+I_{24}H_1+I_{25}H_2\Big)\frac{1}{z^2}\\&+\Big(I_{31}H_1^3+I_{32}H_1^2 H_2+I_{33}H_1 H_2^2+I_{34}H_2^3+I_{35}H_1^2+I_{36}H_1 H_2+I_{37}H_2^2\\
    &+I_{38}H_1+I_{39}H_2\Big)\frac{1}{z^3}+\mathcal{O}(\frac{1}{z^4})\,.
\end{align*}

From the Birkhoff factorization which will appear in Section \ref{IBF}, it is natural to define the following series.
\begin{align*}
    J_{11}&=\frac{I_{11}+I_{11}I_{12}^{\bullet}+I_{21}'+I_{12}^{\bullet}I_{21}'-I_{12}'I_{21}^{\bullet}}{1+I_{11}'-I_{11}^{\bullet}I_{12}'+I_{21}^{\bullet}+I_{11}'I_{12}^{\bullet}}\,,\\
   J_{12}&=\frac{-I_{11}I_{12}'+(1+I_{12}^{\bullet})(I_{12}+I_{22}')-I_{12}'I_{22}^{\bullet}}{1+I_{11}'-I_{11}^{\bullet}I_{12}'+I_{12}^{\bullet}+I_{11}'I_{12}^{\bullet}}\,, \\
   J_{13}&=\frac{-I_{12}I_{12}'+I_{23}'+I_{12}^{\bullet}I_{23}'-I_{12}'I_{23}^{\bullet}}{1+I_{11}'-I_{11}^{\bullet}I_{12}'+I_{12}^{\bullet}+I_{11}'I_{12}^{\bullet}}\,, \\
   J_{14}&=\frac{I_{24}'+I_{12}^{\bullet}I_{24}'-I_{12}'I_{24}^{\bullet}}{1+I_{11}'-I_{11}^{\bullet}I_{12}'+I_{12}^{\bullet}+I_{11}'I_{12}^{\bullet}}\,,\\
   J_{15}&=\frac{I_{25}' + I_{12}^{\bullet} I_{25}' - I_{12}' I_{25}^{\bullet}}{1+I_{11}'-I_{11}^{\bullet}I_{12}'+I_{12}^{\bullet}+I_{11}'I_{12}^{\bullet}}\,,
   \end{align*}
   \begin{align*}
   J_{21}&=\frac{I_{21} + I_{12}^{\bullet} I_{21} + I_{31}' + I_{12}^{\bullet} I_{31}' - I_{12}' I_{31}^{\bullet}}{1+I_{11}'-I_{11}^{\bullet}I_{12}'+I_{12}^{\bullet}+I_{11}'I_{12}^{\bullet}}\,,\\
   J_{22}&=\frac{(1 + I_{12}^{\bullet}) (I_{22} + I_{32}') - I_{12}' (I_{21} + I_{32}^{\bullet})}{1+I_{11}'-I_{11}^{\bullet}I_{12}'+I_{12}^{\bullet}+I_{11}'I_{12}^{\bullet}}\,,\\
   J_{23}&=\frac{(1 + I_{12}^{\bullet}) (I_{23} + I_{33}') - I_{12}' (I_{22} + I_{33}^{\bullet})}{1+I_{11}'-I_{11}^{\bullet}I_{12}'+I_{12}^{\bullet}+I_{11}'I_{12}^{\bullet}}\,,\\
   J_{24}&=\frac{(1 + I_{12}^{\bullet}) I_{34}' - I_{12}' (I_{23} + I_{34}^{\bullet})}{1+I_{11}'-I_{11}^{\bullet}I_{12}'+I_{12}^{\bullet}+I_{11}'I_{12}^{\bullet}}\,,\\
   J_{25}&=\frac{I_{24} + I_{12}^{\bullet} I_{24} + I_{35}' + I_{12}^{\bullet} I_{35}' - I_{12}' I_{35}^{\bullet}}{1+I_{11}'-I_{11}^{\bullet}I_{12}'+I_{12}^{\bullet}+I_{11}'I_{12}^{\bullet}}\,,\\
   J_{26}&=\frac{(1 + I_{12}^{\bullet}) (I_{25} + I_{36}') - I_{12}' (I_{24} + I_{36}^{\bullet})}{1+I_{11}'-I_{11}^{\bullet}I_{12}'+I_{12}^{\bullet}+I_{11}'I_{12}^{\bullet}}\,,\\
   J_{27}&=\frac{(1 + I_{12}^{\bullet}) I_{37}' - I_{12}' (I_{25} + I_{37}^{\bullet})}{1+I_{11}'-I_{11}^{\bullet}I_{12}'+I_{12}^{\bullet}+I_{11}'I_{12}^{\bullet}}\,,\\
   J_{28}&=\frac{I_{38}' + I_{12}^{\bullet} I_{38}' - I_{12}' I_{38}^{\bullet}}{1+I_{11}'-I_{11}^{\bullet}I_{12}'+I_{12}^{\bullet}+I_{11}'I_{12}^{\bullet}}\,,\\
   J_{29}&=\frac{I_{39}' + I_{12}^{\bullet} I_{39}' - I_{12}' I_{39}^{\bullet}}{1+I_{11}'-I_{11}^{\bullet}I_{12}'+I_{12}^{\bullet}+I_{11}'I_{12}^{\bullet}}
\end{align*}
Here, upper subscript $'$ and $^{\bullet}$ mean the action of differential operators $q_1\frac{\partial}{\partial q_1}$ and $q_2\frac{\partial}{\partial q_2}$ respectively.
Define the series $\mathds{J}_{1i}$ by the equations
\begin{align*}
    \mathds{J}_{11}=J_{12}\,,\,\,\,\mathds{J}_{12}=J_{14}\,,\,\,\,\mathds{J}_{13}=J_{15}+ s_1 J_{13}\,,\,\,\,\mathds{J}_{14}=J_{11}-s_2 J_{13}\,.
\end{align*}
Here, $s_1=\beta_0+\beta_1$ and $s_2=\beta_0 \beta_1$. Define the series $K_{ij}$ by the equations
\begin{align*}
    K_{11}&=\frac{-I_{11} I_{11}^{\bullet} - I_{11}^{\bullet} I_{21}' + I_{21}^{\bullet} + I_{11}' I_{21}^{\bullet}}{1+I_{11}'-I_{11}^{\bullet}I_{12}'+I_{12}^{\bullet}+I_{11}'I_{12}^{\bullet}}\,,\\
    K_{12}&=\frac{I_{11} (1 + I_{11}') - I_{11}^{\bullet} (I_{12} + I_{22}') + (1 + I_{11}') I_{22}^{\bullet}}{1+I_{11}'-I_{11}^{\bullet}I_{12}'+I_{12}^{\bullet}+I_{11}'I_{12}^{\bullet}}\,,\\
    K_{13}&=\frac{I_{12} + I_{11}' I_{12} - I_{11}^{\bullet} I_{23}' + I_{23}^{\bullet} + I_{11}' I_{23}^{\bullet}}{1+I_{11}'-I_{11}^{\bullet}I_{12}'+I_{12}^{\bullet}+I_{11}'I_{12}^{\bullet}}\,,\\
    K_{14}&=\frac{-I_{11}^{\bullet} I_{24}' + I_{24}^{\bullet} + I_{11}' I_{24}^{\bullet}}{1+I_{11}'-I_{11}^{\bullet}I_{12}'+I_{12}^{\bullet}+I_{11}'I_{12}^{\bullet}}\,,\\
    K_{15}&=\frac{-I_{11}^{\bullet} I_{25}' + I_{25}^{\bullet} + I_{11}' I_{25}^{\bullet}}{1+I_{11}'-I_{11}^{\bullet}I_{12}'+I_{12}^{\bullet}+I_{11}'I_{12}^{\bullet}}\,,
    \end{align*}
    \begin{align*}
    K_{21}&=\frac{-I_{11}^{\bullet} (I_{21} + I_{31}') + (1 + I_{11}') I_{31}^{\bullet}}{1+I_{11}'-I_{11}^{\bullet}I_{12}'+I_{12}^{\bullet}+I_{11}'I_{12}^{\bullet}}\,,\\
    K_{22}&=\frac{(1 + I_{11}') I_{21} - I_{11}^{\bullet} (I_{22} + I_{32}') + (1 + I_{11}') I_{32}^{\bullet}}{1+I_{11}'-I_{11}^{\bullet}I_{12}'+I_{12}^{\bullet}+I_{11}'I_{12}^{\bullet}}\,,\\
    K_{23}&=\frac{(1 + I_{11}') I_{22} - I_{11}^{\bullet} (I_{23} + I_{33}') + (1 + I_{11}') I_{33}^{\bullet}}{1+I_{11}'-I_{11}^{\bullet}I_{12}'+I_{12}^{\bullet}+I_{11}'I_{12}^{\bullet}}\,,\\
    K_{24}&=\frac{I_{23} + I_{11}' I_{23} - I_{11}^{\bullet} I_{34}' + I_{34}^{\bullet} + I_{11}' I_{34}^{\bullet}}{1+I_{11}'-I_{11}^{\bullet}I_{12}'+I_{12}^{\bullet}+I_{11}'I_{12}^{\bullet}}\,,\\
    K_{25}&=\frac{-I_{11}^{\bullet} (I_{24} + I_{35}') + (1 + I_{11}') I_{35}^{\bullet}}{1+I_{11}'-I_{11}^{\bullet}I_{12}'+I_{12}^{\bullet}+I_{11}'I_{12}^{\bullet}}\,,\\
    K_{26}&=\frac{(1 + I_{11}') I_{24} - I_{11}^{\bullet} (I_{25} + I_{36}') + (1 + I_{11}') I_{36}^{\bullet}}{1+I_{11}'-I_{11}^{\bullet}I_{12}'+I_{12}^{\bullet}+I_{11}'I_{12}^{\bullet}}\,,\\
    K_{27}&=\frac{I_{25} + I_{11}' I_{25} - I_{11}^{\bullet} I_{37}' + I_{37}^{\bullet} + I_{11}' I_{37}^{\bullet}}{1+I_{11}'-I_{11}^{\bullet}I_{12}'+I_{12}^{\bullet}+I_{11}'I_{12}^{\bullet}}\,,\\
    K_{28}&=\frac{-I_{11}^{\bullet} I_{38}' + I_{38}^{\bullet} + I_{11}' I_{38}^{\bullet}}{1+I_{11}'-I_{11}^{\bullet}I_{12}'+I_{12}^{\bullet}+I_{11}'I_{12}^{\bullet}}\,,\\
    K_{29}&=\frac{-I_{11}^{\bullet} I_{39}' + I_{39}^{\bullet} + I_{11}' I_{39}^{\bullet}}{1+I_{11}'-I_{11}^{\bullet}I_{12}'+I_{12}^{\bullet}+I_{11}'I_{12}^{\bullet}}
\end{align*}
Similarly define the series $\mathds{K}_{1i}$ by the equations
\begin{align*}
    \mathds{K}_{11}=K_{12}\,,\,\,\,\mathds{K}_{12}=K_{14}\,,\,\,\,\mathds{K}_{13}=K_{15}+s_1K_{13}\,,\,\,\,\mathds{K}_{14}=K_{11}-s_2K_{13}\,.
\end{align*}
Finally define the series $\mathds{M}_{1i}$ by the equations
\begin{align*}
    \mathds{M}_{11}&=\frac{J_{26}^{\bullet} + \mathds{J}_{12} - \mathds{J}_{11} \mathds{J}_{12}^{\bullet} - \mathds{J}_{13}^{\bullet} \mathds{K}_{11} + (J_{23}^{\bullet} + \mathds{J}_{11}) s_1}{1+\mathds{J}_{11}^{\bullet}}\,,\\
    \mathds{M}_{12}&=\frac{J_{21}^{\bullet} + J_{28}^{\bullet} - \mathds{J}_{12} \mathds{J}_{12}^{\bullet} - I_{11} \mathds{J}_{14}^{\bullet} - \mathds{J}_{13}^{\bullet} \mathds{K}_{12} - (J_{23}^{\bullet} + \mathds{J}_{11}) s_2}{1+\mathds{J}_{11}^{\bullet}}\,,\\
    \mathds{M}_{13}&=\frac{J_{22}^{\bullet} + J_{29}^{\bullet} - \mathds{J}_{12}^{\bullet} \mathds{J}_{13} + \mathds{J}_{14} - I_{12} \mathds{J}_{14}^{\bullet} - 
 \mathds{J}_{13}^{\bullet} \mathds{K}_{13} + (J_{27}^{\bullet} + \mathds{J}_{13}) s_1 + J_{24}^{\bullet} (s_1^2 - s_2)}{1+\mathds{J}_{11}^{\bullet}}\,,\\
 \mathds{M}_{14}&=\frac{J_{25}^{\bullet} - \mathds{J}_{12}^{\bullet} \mathds{J}_{14} - \mathds{J}_{13}^{\bullet} \mathds{K}_{14} - (J_{27}^{\bullet} + \mathds{J}_{13}) s_2 - J_{24}^{\bullet} s_1 s_2}{1+\mathds{J}_{11}^{\bullet}}
\end{align*}
Under the specialization \eqref{P1P1sp}, it is easy to check the following results.
\begin{align*}
    \mathds{J}_{12}=\mathds{J}_{13}=\mathds{K}_{12}=\mathds{K}_{13}=\mathds{M}_{11}=\mathds{M}_{12}=0\,.
\end{align*}

In the next section, we will find the relations between the series defined above.

\subsection{Basic correlators}

\subsubsection{Light markings.} Moduli of quasimaps can be considered with $n$ ordinary (weight 1) markings and k light (weight 0+) markings,
$$\overline{Q}_{g,n|k}^{0+,0+}(\PP^1\times\PP^1,(d_1,d_2))\,.$$
See \cite{BigI} for more explanations.
For $\gamma_i\in H^*_{\mathsf{T}}(\PP^1\times\PP^1)$ and $\delta_j\in H^*_\mathsf{T}([(\CC^2\times\CC^2)/(\CC^*\times\CC^*)])$, we define series for the $K_{\PP^1\times\PP^1}$ geometry,

\begin{multline*}
    \Big\lan \gamma _1\psi  ^{a_1} , \ldots, \gamma _n\psi  ^{a_n} ;  \delta _1, \ldots, \delta _k \Big\ran _{g, n|k, (d_1,d_2)}^{0+, 0+}  = \\
\int _{[\overline{Q}^{0+, 0+}_{g, n|k} (\PP^1\times\PP^1, (d_1,d_2))]^{\vir}} 
e(\text{Obs})\cdot
\prod _{i=1}^n \text{ev}_i^*(\gamma _i)\psi _i ^{a_i} 
\cdot \prod _{j=1}^k \widehat{\text{ev}}_j ^* (\delta _j)\, , 
\end{multline*}
\begin{multline*}
\Big \langle \Big\langle \gamma _1\psi  ^{a_1} , \ldots, \gamma _n\psi  ^{a_n} \Big\rangle\Big\rangle _{0, n}^{0+, 0+} 
= \\ \sum _{d\geq 0}\sum_{k\geq 0} \frac{q_1^{d_1}q_2^{d_2}}{k!}
 \Big\lan    \gamma _1\psi  ^{a_1} , \ldots, \gamma _n\psi  ^{a_n} ; {t}, \ldots, {t}  
 \Big\ran_{0, n|k, (d_1,d_2)}^{0+, 0+} \, ,
 \end{multline*}
 where, in the second series,
 ${t} \in H_{\T}^* ([(\CC^2\times\CC^2)/(\com^*\times\CC^*) ])$.

For each $\mathsf{T}$-fixed point $p_{ij}\in\PP^1\times\PP^1$, let
$$e_{ij}=e(T_{p_{ij}}(\PP^1\times\PP^1))\cdot(-2\alpha_i-2\beta_j)$$
be the equivariant Euler class of the tangent space of $K_{\PP^1\times\PP^1}$ at $p_{ij}$. Let
\begin{align*}
    \phi_{ij}=\frac{(-2\alpha_i-2\beta_j)(H_1-\alpha_{i+1})(H_2-\beta_{j+1})}{e_{ij}}\,,\,\,\phi^{ij}=e_{ij}\phi_{ij}\in H^*_{\mathsf{T}}(\PP^1\times\PP^1)
\end{align*}
be the cohomology classes. The following series will play an important role.
\begin{align*}
    \mathds{S}_{ij}(\gamma):=e_{ij}\lann\frac{\phi_{ij}}{z-\psi},\gamma \rann^{0+,0+}_{0,2}
\end{align*}
 We also write $$\mathds{S}(\gamma):=\sum_{i=0}^1\sum_{j=0}^1 {\phi_{ij}} \mathds{S}_{ij}(\gamma)\, .$$

\subsubsection{I-function and Birkhoff factorization.}\label{IBF}
Via the geometry of weighted quasimap graph space
\begin{equation*}
     \mathsf{QG}_{g, n|k, (d_1,d_2) }^{0+,0+} ([(\CC^2\times\CC^2)/(\com^*\times\CC^*)] ) \, 
\end{equation*}
the big $I$-function is defined in \cite{BigI}. See also \cite[Section 3.4]{LP} for brief introduction.

The $\mathds{I}$-function can be evaluated explicitly using the arguments in \cite[Section 5]{BigI}.

\begin{Prop}
For $\mathsf{t}=t_1H_1+t_2 H_2\in H^*_{\mathsf{T}}([(\CC^2\times\CC^2)/(\CC^*\times\CC^*)],\QQ)$,
\begin{multline}
    \mathds{I}(\mathsf{t})=\sum_{d_1,d_2\ge 0}q_1^{d_1}q_2^{d_2}e^{t_1(H_1+d_1z)/z+t_2(H_2+d_2z)/z}\\
    \cdot\frac{\prod_{k=0}^{2d_1+2d_2-1}(-2H_1-2H_2-kz)}{\prod_{i=0}^1\prod_{k=1}^{d_1}(H_1-\alpha_i+kz)\prod_{j=0}^1\prod_{k=1}^{d_2}(H_1-\beta_j+kz)}\,.
\end{multline}
\end{Prop}

Using Birkhoff factorization, an evaluation of the series $\mathds{S}(H_1^iH_2^j)$ can be obtained from $\mathds{I}$-function, see \cite{KL}:

\begin{align}\label{EofS}
    \nonumber\mathds{S}(1)=&\mathds{I}\,,\\
    \mathds{S}(H_1)=&E_{11}\cdot z\frac{\partial}{\partial t_1}\mathds{S}(1)+E_{12}\cdot z\frac{\partial}{\partial t_2}\mathds{S}(1)\,,\\
    \nonumber\mathds{S}(H_2)=&E_{21}\cdot z\frac{\partial}{\partial t_1}\mathds{S}(1)+E_{22}\cdot z\frac{\partial}{\partial t_2}\mathds{S}(1)\,,\\
    \nonumber\mathds{S}(H_1 H_2)=&E_{31}\cdot  z\frac{\partial}{\partial t_1}\mathds{S}(H_2)+E_{32}\cdot \mathds{S}(1)\,.
\end{align}
Here, $E_{ij}$ are the series defined by
\begin{align*}
    E_{11}=&\frac{1+I_{12}^{\bullet}}{1+I_{11}'-I_{11}^{\bullet}I_{12}'+I_{12}^{\bullet}+I_{11}'I_{12}^{\bullet}}\,,\\
    E_{12}=&\frac{-I_{12}'}{1+I_{11}'-I_{11}^{\bullet}I_{12}'+I_{12}^{\bullet}+I_{11}'I_{12}^{\bullet}}\,,\\
    E_{21}=&\frac{-I_{11}^{\bullet}}{1+I_{11}'-I_{11}^{\bullet}I_{12}'+I_{12}^{\bullet}+I_{11}'I_{12}^{\bullet}}\,,\\
    E_{22}=&\frac{1+I_{11}'}{1+I_{11}'-I_{11}^{\bullet}I_{12}'+I_{12}^{\bullet}+I_{11}'I_{12}^{\bullet}}\,,\\
    E_{31}=&\frac{1}{1+\mathds{K}_{11}'}\,,\\
    E_{32}=&\frac{-\mathds{K}_{14}'}{1+\mathds{K}_{11}'}\,.
\end{align*}

\subsubsection{Picard-Fuchs equations and asymptotic expansion.}

The function $\mathds{I}$ satisfies the following Picard-Fuchs equations.

\begin{align*}
    &\Big(\Big(z\frac{d}{dt_1}\Big)^2-1-q_1\Big(2\Big(z\frac{d}{dt_1}\Big)+2\Big(z\frac{d}{dt_2}\Big)\Big)\Big(2\Big(z\frac{d}{dt_1}\Big)+2\Big(z\frac{d}{dt_2}\Big)+z\Big)\Big)\mathds{I}=0\,,\\
    &\Big(\Big(z\frac{d}{dt_2}\Big)^2+1-q_2\Big(2\Big(z\frac{d}{dt_1}\Big)+2\Big(z\frac{d}{dt_2}\Big)\Big)\Big(2\Big(z\frac{d}{dt_1}\Big)+2\Big(z\frac{d}{dt_2}\Big)+z\Big)\Big)\mathds{I}=0
\end{align*}

Define small $I$-function
$$\mathds{I}(q_1,q_2)\in  H^{*}_{\mathds{T}}(\PP^1\times\PP^1,\QQ)[[q_1,q_2]]$$

by the restriction 

$$\overline{\mathds{I}}(q_1,q_2)=\mathds{I}(q_1,q_2,t_1,t_2)|_{t_1=t_2=0}$$

Define differential operators
$$\mathsf{D}_1=q_1 \frac{d}{dq_1}\,,\,\,\mathsf{D}_2=q_2\frac{d}{dq_2}\,\,,\,M_1=H_1+z\mathsf{D}_1\,,\,\,M_2=H_2+z\mathsf{D}_2\,.$$

The small $I$-function satisfies following Picard-Fuchs equations.
\begin{align}\label{PFE}
    &\Big(M_1^2-1-q_1(2M_1+2M_2)(2M_1+2M_2+z)\Big)\overline{\mathds{I}}=0\,,\\ 
    \nonumber&\Big(M_2^2+1-q_1(2M_1+2M_2)(2M_1+2M_2+z)\Big)\overline{\mathds{I}}=0\,.
\end{align}

The restriction $\overline{\mathds{I}}|_{H_1=\alpha_i,H_2=\beta_j}$ admits following asymptotic form
\begin{align}\label{ASY}
    \overline{\mathds{I}}|_{H_1=\alpha_i,H_2=\beta_j}=e^{\mathds{U}_{ij}/z}\Big(\mathds{R}_{0,ij}+\mathds{R}_{1,ij}z+\mathds{R}_{2,ij}z^2+\dots\Big)
\end{align}
with series $\mathds{U}_{ij},\,\mathds{R}_{k,ij}\in \CC[[q_1,q_2]]$.
Define series $\mathds{L}_{ij}$ and $\mathds{UD}_{ij}$ for $0\le i\le 1,\,0\le j \le 1$ by
\begin{align*}
    \mathds{L}_{ij}=\alpha_i+q_1\frac{d}{dq_1}\mathds{U}_{ij}\,,\,\,\mathds{UD}_{ij}=\beta_j+q_2\frac{d}{dq_2}\mathds{U}_{ij}\,.
\end{align*}

Let $L_{ij}$ be the series in $q_1$ defined by constant term with respect to $q_2$. The series $L_{ij}$ is found by solving differential equations obtained from the coefficient of $z^k$. 
\begin{align}\label{L}
L_{ij}(q_1)=\frac{4 \beta_j q_1+\alpha_i\sqrt{1+4(-1+\beta_j^2)q_1}}{1-4q_1}\,.
\end{align}
Define the series $u_{k,ij}$ and $R_{nk,ij}$ by the following equations.
\begin{align}\label{smse}
    \mathds{U}_{ij}&:=u_{0,ij}+u_{1,ij} q_2+u_{2,ij} q_2^2+\dots\,,\\
    \mathds{R}_{n,ij}&:=R_{n0,ij}+R_{n1,ij}q_2+R_{n2,ij}q_2^2+\dots\,.
\end{align}
Denote by $\mathsf{G}_{ij}$ the subring of $\CC[[q_1]]$ generated by $L_{ij}$ and  $\frac{1}{\sqrt{1+\beta_j L_{ij}}}$\,

$$\mathsf{G}_{ij}=\CC[L_{ij},\frac{1}{\sqrt{1+\beta_j L_{ij}}}]\subset \CC[[q_1]]\,.$$
By analyzing the differential equations from the coefficient of $z^k$, we obtain the following results.

\begin{Lemma}\label{FG1}
We have
\begin{align*}
    \mathds{L}_{ij}\,,\,\,\mathds{UD}_{ij}\,,\,\,\mathds{R}_{n,ij}\in \mathsf{G}_{ij}[[q_2]]\,.
\end{align*}
\end{Lemma}

\begin{proof}
To simplify the notations, we will only prove the lemma for the case $(i,j)=(0,0)$ and omit the index $(i,j)$ for the series $\mathds{L}_{ij},\,\mathds{UD}_{ij},\,\mathds{R}_{k,ij}$. The proof for other $(i,j)$ follows from the same argument.
If we apply the equations \eqref{PFE} to the asymptotic form \eqref{ASY}, the coefficients of $z$ in each equation yield the following equation,
\begin{align}
    1 + \mathds{UD}^2 - 4 q_2 (\mathds{L} + \mathds{UD})^2=0\,.
\end{align}
By applying \eqref{smse} to above equation, we obtain the equation $\text{Eq}_k$ from the coefficient of $q_2^k$. Each term of equation $\text{Eq}_k$ is a monomial of $(\mathsf{D}_1)^l L$, $(\mathsf{D}_1)^lu_m$ for $l\ge 0$. Furthermore one can easily check that the equations $\text{Eq}_k$ are linear in $u_k$ with coefficient $2k\sqrt{-1}$ and has no $(\mathsf{D}_1)^lu_m$ term for $m \ge k+1$.
The statements of the lemma for $\mathds{L}$ and $\mathds{UD}$ follow from following equation which is easy to check from \eqref{L},
\begin{align*}
    \mathsf{D}_1 L=\frac{(\sqrt{-1} + L) (-1 + L^2)}{2 (1 + \sqrt{-1} L)}\,.
\end{align*}

The same argument for the coefficient of $z^{n+1}$ in \eqref{PFE} gives the proof of the statement of the lemma for $\mathds{R}_n$.

\end{proof}

\subsubsection{Relations on generators.}\label{RoG1}
We prove some relations on the generators defined in Section \ref{Gen}.

\begin{Prop}\label{MG}
The series 
\begin{multline*}
\mathds{J}_{11}',\,\mathds{J}_{11}^{\bullet},\,\mathds{J}_{14}',\,\mathds{J}_{14}^{\bullet},\,\mathds{K}_{11}',\,\mathds{K}_{11}^{\bullet},\,
\mathds{K}_{14}',\,\mathds{K}_{14}^{\bullet},\,\mathds{M}_{12}',\,\mathds{M}_{12}^{\bullet},\,\mathds{M}_{13}',\,\mathds{M}_{13}^{\bullet}
\end{multline*}
are representable as rational functions in $I_{11}',\,I_{11}^{\bullet},\,\mathds{L},\,\mathds{UD}$. 
\end{Prop}

\begin{proof}

First consider the quantum product by 
\begin{align*}
    (H_1+q_1\frac{\partial}{\partial q_1})\,\,,(H_2+q_2\frac{\partial}{\partial q_2}).
\end{align*}
In the basis $\{1,H_1,H_2,H_1 H_2\}$, they can be considered as matrix multiplications and the eigenvalues of these matrices are $\mathds{L}_{ij}$ and $\mathds{UD}_{ij}$, respectively.

From the commutativity of quantum product, we also get some relations.
Finally, using the fact that $q_1\frac{\partial}{\partial q_1}$ and $q_2\frac{\partial}{\partial q_2}$ commute, we get some relations by repeatedly applying $q_1\frac{\partial}{\partial q_1}$ and $q_2\frac{\partial}{\partial q_2}$ to $\mathds{S}$-operators.

By the relations we get from above, we can find following explicit equations. Since the result are independent of the index $(i,j)$, we will omit the index $(i,j)$ from the series $\mathds{L}_{ij}$ and $\mathds{UD}_{ij}$.
\begin{align*}
    \mathds{J}_{11}'&=-\frac{(\mathds{L} + \mathds{UD})^2 (-1 + \mathds{L}^2 + I_{11}^{\bullet} (-1 + \mathds{L}^2) - I_{11}' \mathds{L}\cdot \mathds{UD})}   {
 2 (1 + I_{11}' + I_{11}^{\bullet})^2 (-1 + \mathds{L}^2 - \mathds{UD}^2)}\,,\\
 \end{align*}
 \begin{align*}
 \mathds{J}_{11}^{\bullet}&=\frac{1}{2 (1 + I_{11}' + I_{11}^{\bullet})^2 (-1 + \mathds{L}^2 - \mathds{UD}^2)}\Big(2 - 2 \mathds{L}^2 - \mathds{L}^3 \mathds{UD} \\&+ 2 \mathds{UD}^2 - 2 \mathds{L}^2 \mathds{UD}^2 - \mathds{L}\cdot \mathds{UD}^3 + 
 (I_{11}')^2 (2 - 2 \mathds{L}^2 + 2 \mathds{UD}^2) \\
 &+ (I_{11}^{\bullet})^2 (2 - 2 \mathds{L}^2 + 2 \mathds{UD}^2) - 
 I_{11}^{\bullet} (\mathds{L}^3 \mathds{UD} + \mathds{L} \cdot\mathds{UD}^3 \\&- 4 (1 + \mathds{UD}^2) + 2 \mathds{L}^2 (2 + \mathds{UD}^2)) + 
 I_{11}' (4 + 5 \mathds{UD}^2 + \mathds{UD}^4\\&
 + \mathds{L}^2 (-3 + \mathds{UD}^2) + 
    I_{11}^{\bullet} (4 - 4 \mathds{L}^2 + 4 \mathds{UD}^2) + 2 \mathds{L} (\mathds{UD} + \mathds{UD}^3))\Big)\,,\\
 \end{align*}
 \begin{align*}
 \mathds{J}_{14}'&=\frac{1}{2 (1 + I_{11}' + I_{11}^{\bullet})^2 (-1 + \mathds{L}^2 - \mathds{UD}^2)}\Big((1 + I_{11}^{\bullet}) (-1 + \mathds{L}^2) (-2 + \mathds{L}^2 \\
&- \mathds{UD}^2 + 2 I_{11}^{\bullet} (-1 + \mathds{L}^2 + \mathds{L}\cdot \mathds{UD})) + 
 I_{11}' (-3 \mathds{L}^3 \mathds{UD} + 4 (1 + \mathds{UD}^2) \\&- 2 \mathds{L}^2 (2 + \mathds{UD}^2) + 
    \mathds{L}\cdot \mathds{UD} (4 + \mathds{UD}^2) - 4 I_{11}^{\bullet} (-1 + \mathds{L}^2) (1 + \mathds{L}\cdot \mathds{UD} \\&+ \mathds{UD}^2)) + 
 2 (I_{11}')^2 (1 + \mathds{UD}^2 + \mathds{L}^2 (-1 + \mathds{UD}^2) + \mathds{L} (\mathds{UD} + \mathds{UD}^3))\Big)\,,\\
 \end{align*}
 \begin{align*}
 \mathds{J}_{14}^{\bullet}&=\frac{1}{2 (1 + I_{11}' + I_{11}^{\bullet})^2 (-1 + \mathds{L}^2 - \mathds{UD}^2)}\Big((1 + I_{11}^{\bullet}) \mathds{L} \mathds{UD} (-2 + \mathds{L}^2 \\&- \mathds{UD}^2 + 2 I_{11}^{\bullet} (-1 + \mathds{L}^2 + \mathds{L} \cdot\mathds{UD})) + 
 2 (I_{11}')^2 (\mathds{L}^2 + \mathds{UD}^2 + \mathds{UD}^4 \\&+ \mathds{L} (\mathds{UD} + \mathds{UD}^3)) + 
 I_{11}' (\mathds{UD}^2 + \mathds{UD}^4 + \mathds{L}^2 (1 - (3 + 4 I_{11}^{\bullet}) \mathds{UD}^2) \\&- 
    2 (1 + 2 I_{11}^{\bullet}) \mathds{L} (\mathds{UD} + \mathds{UD}^3))\Big)\,,\\
 \end{align*}
 \begin{align*}
 \mathds{K}_{11}'&=\frac{1}{2 (1 + I_{11}' + I_{11}^{\bullet})^2 (-1 + \mathds{L}^2 - \mathds{UD}^2)}\Big(2 - 2 \mathds{L}^2 - \mathds{L}^3 \mathds{UD} + 2 \mathds{UD}^2 \\&- 2 \mathds{L}^2 \mathds{UD}^2 - \mathds{L}\cdot \mathds{UD}^3 + 
 (I_{11}')^2 (2 - 2 \mathds{L}^2 + 2 \mathds{UD}^2) + (I_{11}^{\bullet})^2 (2 - 2 \mathds{L}^2 \\&+ 2 \mathds{UD}^2) + 
 I_{11}^{\bullet} (4 + \mathds{L}^4 - 2 \mathds{L} \cdot\mathds{UD} + 2 \mathds{L}^3 \mathds{UD} + 3 \mathds{UD}^2 + \mathds{L}^2 (-5 + \mathds{UD}^2)) \\&- 
 I_{11}' (\mathds{L}^3 \mathds{UD} + \mathds{L} \cdot\mathds{UD}^3 + 4 I_{11}^{\bullet} (-1 + \mathds{L}^2 - \mathds{UD}^2) - 4 (1 + \mathds{UD}^2) \\&+ 
    2 \mathds{L}^2 (2 + \mathds{UD}^2))\Big) \,,\\   
 \end{align*}
 \begin{align*}
 \mathds{K}_{11}^{\bullet}&=-\frac{1}{2 (1 + I_{11}' + I_{11}^{\bullet})^2 (-1 + \mathds{L}^2 - \mathds{UD}^2)}\Big((\mathds{L} + \mathds{UD})^2 (1 + I_{11}' \\&- I_{11}^{\bullet} \mathds{L}\cdot \mathds{UD} + \mathds{UD}^2 + I_{11}' \mathds{UD}^2)\Big)\,\\
 \end{align*}
 \begin{align*}
 \mathds{K}_{14}'&=\frac{1}{2 (1 + I_{11}' + I_{11}^{\bullet})^2 (-1 + \mathds{L}^2 - \mathds{UD}^2)}\Big(-2 (I_{11}^{\bullet})^2 (-\mathds{L}^2 + \mathds{L}^4 \\&- \mathds{L}\cdot \mathds{UD} + \mathds{L}^3 \mathds{UD} - \mathds{UD}^2) - (1 + I_{11}') \mathds{L}\cdot \mathds{UD} (2 - 
    \mathds{L}^2 + \mathds{UD}^2 \\&+ 2 I_{11}' (1 + \mathds{L}\cdot \mathds{UD} + \mathds{UD}^2)) + 
 I_{11}^{\bullet} (-\mathds{L}^4 + \mathds{UD}^2 - 2 \mathds{L} (\mathds{UD} \\&+ 2 I_{11}' \mathds{UD}) + 2 \mathds{L}^3 (\mathds{UD} + 2 I_{11}' \mathds{UD}) + 
    \mathds{L}^2 (1 + (3 + 4 I_{11}') \mathds{UD}^2))\Big)\,,\\
 \end{align*}
 \begin{align*}
 \mathds{K}_{14}^{\bullet}&=\frac{1}{2 (1 + I_{11}' + I_{11}^{\bullet})^2 (-1 + \mathds{L}^2 - \mathds{UD}^2)}\Big((-2 + \mathds{L}^2 - \mathds{UD}^2) (1 + \mathds{UD}^2) \\&- 2 (I_{11}')^2 (1 + \mathds{UD}^2) (1 + \mathds{L}\cdot \mathds{UD} + \mathds{UD}^2) +
  I_{11}' (1 + \mathds{UD}^2) (-4 + \mathds{L}^2\\& - 2 \mathds{L}\cdot \mathds{UD} - 3 \mathds{UD}^2 + 
    4 I_{11}^{\bullet} (-1 + \mathds{L}^2 + \mathds{L} \cdot\mathds{UD})) - 
 2 (I_{11}^{\bullet})^2 (1 - \mathds{L}\cdot \mathds{UD} \\&+ \mathds{L}^3 \mathds{UD} + \mathds{UD}^2 + \mathds{L}^2 (-1 + \mathds{UD}^2)) + 
 I_{11}^{\bullet} (-\mathds{L}^3 \mathds{UD} - 4 (1 + \mathds{UD}^2)\\& + 2 \mathds{L}^2 (2 + \mathds{UD}^2) + \mathds{L} \cdot\mathds{UD} (4 + 3 \mathds{UD}^2))\Big)  \,, 
\end{align*}
\begin{align*}
    \mathds{M}_{12}'&=-\frac{1}{(\mathds{L} + \mathds{UD})^2}\Big((1 + 3 I_{11}' + 2 (I_{11}')^2) \mathds{L}^2 + 
 2 (I_{11}' + (I_{11}')^2 \\&- 2 I_{11}' I_{11}^{\bullet} - I_{11}^{\bullet} (2 + I_{11}^{\bullet})) \mathds{L}\cdot \mathds{UD} + (1 + 
    3 I_{11}' + 2 (I_{11}')^2) \mathds{UD}^2\Big)\,,
\end{align*}
\begin{align*}
    \mathds{M}_{12}^{\bullet}&=\frac{1}{(\mathds{L} + \mathds{UD})^2}\Big((1 + I_{11}^{\bullet} + 2 (I_{11}^{\bullet})^2) \mathds{L}^2 - 
 2 ((I_{11}')^2 - (-1 + I_{11}^{\bullet}) I_{11}^{\bullet}\\& + 2 I_{11}' (1 + I_{11}^{\bullet})) \mathds{L}\cdot \mathds{UD} + (1 + I_{11}^{\bullet} + 
    2 (I_{11}^{\bullet})^2) \mathds{UD}^2\Big)\,,
\end{align*}
\begin{align*}
    \mathds{M}_{13}'&=-\frac{1}{(\mathds{L} + \mathds{UD})^2}\Big((1 + I_{11}' + 2 (I_{11}')^2) \mathds{L}^2 - 
 2 (I_{11}' - (I_{11}')^2 + 2 I_{11}' I_{11}^{\bullet} \\&+ I_{11}^{\bullet} (2 + I_{11}^{\bullet})) \mathds{L}\cdot \mathds{UD} + (1 + I_{11}' + 
    2 (I_{11}')^2) \mathds{UD}^2\Big)\,,
\end{align*}
\begin{align*}
    \mathds{M}_{13}^{\bullet}&=\frac{1}{(\mathds{L} + \mathds{UD})^2}\Big((1 + 3 I_{11}^{\bullet} + 2 (I_{11}^{\bullet})^2) L^2 - 
 2 ((I_{11}')^2 + 2 I_{11}' (1 + I_{11}^{\bullet})\\& - I_{11}^{\bullet} (1 + I_{11}^{\bullet})) \mathds{L}\cdot \mathds{UD} + (1 + 
    3 I_{11}^{\bullet} + 2 (I_{11}^{\bullet})^2) \mathds{UD}^2\Big)\,.
\end{align*}

\end{proof}

The proof of Proposition \ref{MG} yields more relations among which the following relation will be needed.

\begin{align}\label{Re}
    &(1 + I_{11}') \mathds{L} (\mathds{L} - \mathds{UD}) \mathds{UD} (\mathds{L} + \mathds{UD})^2 (1 - 4 I_{11}^{\bullet\bullet} \mathds{L}\cdot \mathds{UD} + \mathds{UD}^2 + 
    I_{11}' (1 + \mathds{UD}^2) \\\nonumber&+ 4 I_{11}'^{\bullet} (1 + \mathds{UD}^2)) - 
 (I_{11}^{\bullet})^2 (\mathds{UD}^3 + \mathds{UD}^5 + \mathds{L}^5 (1 + \mathds{UD}^2) - \mathds{L}^3 (1 + \mathds{UD}^2)^2 \\\nonumber&+ 
    \mathds{L}^4 (\mathds{UD} + \mathds{UD}^3) + \mathds{L} (\mathds{UD}^2 + \mathds{UD}^4) - \mathds{L}^2 \mathds{UD} (-3 + 6 \mathds{UD}^2 + \mathds{UD}^4)) \\\nonumber&+
  I_{11}^{\bullet} (4 I_{11}^{\bullet\bullet} \mathds{L}^6 \mathds{UD} - 
    4 (I_{11}'^{\bullet} + I_{11}^{\bullet\bullet}) \mathds{L}^4 \mathds{UD} (1 + \mathds{UD}^2) - (1 + I_{11}' + 
       4 I_{11}'^{\bullet}) \mathds{UD}^3 \\\nonumber&(1 + \mathds{UD}^2) + 
    \mathds{L}^3 (1 + \mathds{UD}^2) (1 + I_{11}' + 2 \mathds{UD}^2 + 2 I_{11}' \mathds{UD}^2 - 4 I_{11}^{\bullet\bullet} \mathds{UD}^2 \\\nonumber&+ 
       4 I_{11}'^{\bullet} (1 + \mathds{UD}^2)) - 
    \mathds{L}^5 (1 + \mathds{UD}^2 - 4 I_{11}^{\bullet\bullet} \mathds{UD}^2 + I_{11}' (1 + \mathds{UD}^2) \\\nonumber&+ 
       4 I_{11}'^{\bullet} (1 + \mathds{UD}^2)) - 
    \mathds{L}\cdot \mathds{UD}^2 (1 + 2 \mathds{UD}^2 - 4 I_{11}^{\bullet\bullet} \mathds{UD}^2 + \mathds{UD}^4 + 4 I_{11}'^{\bullet} (1 + \mathds{UD}^2) \\\nonumber&+ 
       I_{11}' (1 + \mathds{UD}^2)^2) + 
    \mathds{L}^2 \mathds{UD} (1 + 4 I_{11}^{\bullet\bullet} + \mathds{UD}^2 + I_{11}' (1 + \mathds{UD}^2) \\\nonumber&+ 
       4 I_{11}'^{\bullet} (1 + \mathds{UD}^2)^2))=0\,.
\end{align}
Here, we also omit the index $(i,j)$ from the series $\mathds{L}_{ij}$ and $\mathds{UD}_{ij}$.

Define series $a_k(q_1)$ by the following equation.
\begin{align}\label{Is}
    I_{11}=a_0(q_1)+a_1(q_1)q_2+a_2(q_1)q_2^2+\dots\,.
\end{align}
\noindent The series $a_0$ satisfies following equation.
\begin{align*}
    1+\mathsf{D}_1a_0=(1-4q)^{-1/2}\,.
\end{align*}
Define a new series 
\begin{align*}
    r_{ij}:=\sqrt{L_{ij}}\,.
\end{align*}
\noindent The series $r_{ij}$ satisfies following differential equation.
\begin{align}\label{Dr}
    \mathsf{D}_1r_{ij}=\frac{(\sqrt{-1} + r_{ij}^2) (-1 + r_{ij}^4)}{4 (r + \sqrt{-1} r_{ij}^3)}\,.
\end{align}
From \eqref{L}, we obtain the following equation.
\begin{align}\label{Da0}
    1+\mathsf{D}_1 a_0=\frac{(1+\sqrt{-1})+(1-\sqrt{-1})r_{ij}^2}{2r_{ij}}
\end{align}

\begin{Lemma}\label{FG2}
For $k\ge 1$, We have
\begin{align*}
    a_k \in r_{ij}\cdot\CC[r_{ij}^2,\,r_{ij}^{-2},\,,(1+\beta_j r_{ij}^2)^{-1}]\,.
\end{align*}
\end{Lemma}

\begin{proof}
If we apply the forms \eqref{ASY} and \eqref{Is} to the equation \eqref{Re}, we get the equation $\text{Eq}_k$ from the coefficient of $q_2^k$. Each term of $\text{Eq}_k$ is monomial of $(\mathsf{D}_1)^lL$, $(\mathsf{D}_1)^la_m$, $(\mathsf{D}_1)^lu_m$. Furthermore $\text{Eq}_k$ is linear in $a_k$ with coefficient $$(2k)^2(1+\mathsf{D}_1a_0)L_{ij}^2(-\sqrt{-1}+L_{ij})(\sqrt{-1}+L_{ij})^2$$ and has no $(\mathsf{D}_1)^la_m$ with $m \ge k+1$. The statement of the proposition follows from \eqref{Dr} and \eqref{Da0}.
\end{proof}

\subsection{Higher genus series}

\subsubsection{Graphs.} Let $ g \ge 2$.
A {\em decorated graph} $\Gamma\in\mathsf{G}_g$ consist of the data $(\mathsf{V},\, \mathsf{E},\, \mathsf{g},\,\mathsf{p} )$ such that
\begin{itemize}
 \item[(i)] $\mathsf{V}$ is the vertex set,
 \item[(ii)] $\mathsf{E}$ is the edge set (possibly including self-edges),
 \item[(iii)] $\mathsf{g}:\mathsf{V}\rightarrow \ZZ_{\ge 0}$ is a genus assignment satisfying
 $$g=\sum_{v\in\mathsf{V}}\mathsf{g}(v)+h^1(\Gamma)$$ and for which $(\mathsf{V},\mathsf{E},\mathsf{g})$ is stable graph,
 \item[(iv)] $\mathsf{p}:\mathsf{V}\rightarrow (\PP^1\times\PP^1)^{\mathsf{T}}$ is an assignment of a $\mathsf{T}$-fixed point $\mathsf{p}(v)$ to each vertex $v\in\mathsf{V}$.
\end{itemize}

\subsubsection{Localization formula.}
We summarize the localization formula for the $K_{\PP^1\times\PP^1}$ quasimap theories.

We write the localization formula as
\begin{align*}
    \sum_{d_1,d_2\ge 0}[\overline{Q}_{g}(K_{\PP^1\times\PP^1},(d_1,d_2))]^{\text{vir}}q_1^{d_1}q_2^{q_2}=\sum_{\Gamma\in\mathsf{G}_g}\text{Cont}_{\Gamma}\,.
\end{align*}

\begin{Prop}\label{DeCo}
We have
\begin{align*}
    \text{\em Cont}_{\Gamma}=\frac{1}{|\text{\em Aut}(\Gamma)|}\sum_{\mathsf{A}\in\ZZ^{\mathsf{F}}_{>0}}\prod_{v\in\mathsf{V}}\text{\em Cont}^{\mathsf{A}}_{\Gamma}(v)\prod_{e\in\mathsf{E}}\text{\em Cont}^{\mathsf{A}}_{\Gamma}(e)\,,
\end{align*}
where the vertex and edge contributions with incident flag $\mathsf{A}$-values $(a_1,\dots,a_n)$ and $(b_1,b_2)$ respectively are
\begin{align*}
    \text{\em Cont}^{\mathsf{A}}_{\Gamma}(v)=\mathsf{P}\Big[\psi_1^{a_1-1},\dots,\psi_n^{a_n-1}|\mathsf{H}^{\mathsf{p}(v)}_{\mathsf{g}(v)}\Big]^{\mathsf{p}(v),0+}_{\mathsf{g}(v),n}\,,
\end{align*}
\begin{multline*}
    \text{\em Cont}^{\mathsf{A}}_{\Gamma}(e)=\\\Big[(-1)^{b_1+b_2}e^{-\frac{\mathds{U}_{\mathsf{p}_1}}{z}-\frac{\mathds{U}_{\mathsf{p}_2}}{y}}\sum_{0\le i\le1,\,0\le j\le1}\overline{\mathds{S}}_{\mathsf{p}_1}(\phi_{ij})|_{z=x}\overline{\mathds{S}}_{\mathsf{p}_2}(\phi^{ij})|_{z=y}\Big]_{x^{b_1}y^{b_2-1}-x^{b_1+1}y^{b_2-2}+\dots+(-1)^{b_1-1}x^{b_1+b_2-1}}\,.
\end{multline*}

\end{Prop}

For the precise definition of notations in the vertex contribution, see \cite{LP}. In this paper we do not need the exact definition of vertex contribution. We only need following results for vertex contribution
$$\text{Cont}^{\mathsf{A}}_{\Gamma}(v)\in \CC[\mathds{R}_{0,ij}^{-1},\mathds{R}_{1,ij},\mathds{R}_{2,ij},\dots]\,,$$
which follows from the definition.

The subscript in the edge contribution signifies a (signed) sum of the respective coefficients. 

\subsubsection{Proof of Theorem \ref{MT1}}\label{Mainpf}

In the proof of Lemma \ref{FG1}, we can actually check that $\frac{1}{\sqrt{1+\alpha_i L_{ij}}}$ only appears in $\mathds{R}_{k,ij}$. And we can easily check in the decomposition formula in Proposition \ref{DeCo} that the order of factor $\frac{1}{\sqrt{1+\alpha_i L_{ij}}}$ in $\mathcal{F}^{\mathsf{SQ}}_g$ is even. (Precisely, $2\mathsf{g}(v)-2$ at the vertex $v$.)
By Lemma \ref{FG1}, \eqref{EofS} and Lemma \ref{FG2}, Proposition \ref{DeCo} immediately yields

\begin{multline*}
    \mathcal{F}^{\mathsf{SQ}}_g\in\CC[[q_2]][r_{00},r_{00}^{-1},(1+\beta_{0} r_{01}^2)^{-1},r_{01},r_{01}^{-1},(1+\beta_{1} r_{01}^2)^{-1},\\r_{10},r_{10}^{-1},(1+\beta_{0} r_{10}^2)^{-1},r_{11},r_{11}^{-1},(1+\beta_{1} r_{11}^2)^{-1}]\,.
\end{multline*}
 
Now we need to show that the order of $r_{ij}$ in $\mathcal{F}^{\mathsf{SQ}}_g$ is always even. By \eqref{EofS} we can represent each edge contribution in the Proposition \ref{DeCo} formally as rational function in variables $\mathsf{D}_1^m\mathsf{D}_2^n I_{11}$, $\mathsf{D}_1^m\mathsf{D}_2^n \mathds{U}_{k,ij}$ and $\mathsf{D}_1^m\mathsf{D}_2^n \mathds{R}_{k,ij}$. By direct calculations we can check that the sum of order of the factors $\mathsf{D}_1^m\mathsf{D}_2^n I_{11}$ is always even. Therefore we conclude that order of $r_{ij}$ in $\mathcal{F}^{\mathsf{SQ}}_g$ in each edge contribution is always even from \eqref{Dr}, \eqref{Da0} and Lemma \ref{FG2}.

Now from the definitions of $r_{ij}$, we obtain 
\begin{align}\label{FGg}
\mathcal{F}^{\mathsf{SQ}}_g\in\CC[X,(\sqrt{-1}+X)^{-1},(\sqrt{-1}-X)^{-1}]\,.
\end{align}
Since all coefficients of $q_1^{k_1}q_2^{k_2}$ in $\mathcal{F}^{\mathsf{SQ}}_g$ are real numbers by definition, the statement of the Theorem \ref{MT1}
$$\mathcal{F}^{\mathsf{SQ}}_g\in\CC[X,(1+X^2)^{-1}]\,.$$
follows immediately from \eqref{FGg}.

To finish the proof, we need to show that 
$$\mathcal{F}^{\mathsf{SQ}}_g\in\CC[X^2,(1+X^2)^{-1}]\,.$$
For this, we need the mirror symmetry argument which was explained by Iritani.

Using the Givental's equivariant mirror $K_{\PP^1\times\PP^1}$,  the asymptotic form \eqref{ASY} of $\overline{\mathds{I}}|_{H_1=\alpha_i,H_2=\beta_j}$ can be calculated using the oscillatory integral associated to the mirror of $K_{\PP^1\times\PP^1}$. See \cite[Proposition 6.9]{CCIT2} for the precise statement. Using this oscillatory integral, it is easy to see that the presentation of $\mathds{R}_{n,ij}$ in terms of $L_{ij}$ do not depend on the choice of $(i,j)$. In other words, $\mathds{R}_{n,ij}$ in \eqref{ASY} have same polynomial expressions in terms of $L_{ij}$ for all $(i,j)$. For example, this was explained explicitly for the case of $K_{\PP^2}$ in \cite[Appendix A]{LP3}. The argument in \cite[Appendix A]{LP3} applies to the case of $K_{\PP^1\times\PP^1}$ to yield the similar result. From this observation and Proposition \ref{DeCo}, we conclude that $\mathcal{F}^{\mathsf{SQ}}_g$ is symmetric rational function in $L_{ij}$. Then the statement of the Theorem

$$\mathcal{F}^{\mathsf{SQ}}_g\in\CC[X^2,(1+X^2)^{-1}]\,$$
follows from the fact that $L_{0i}$ and $L_{1i}$ are the two roots of the equations for $i=0,1$,
$$L_i^2-\frac{(-1)^i\sqrt{-1} 8 q_1 }{1-4q_1}L_i-1=0\,.$$

\section{Elliptic fibration : Surface}\label{MS2}

\subsection{Overview.}
We study the $((3,2),(0,0))$-twisted theory on $\PP^2\times\PP^1$. This theory recover the standard theory of K3 surface X, defined by the general section of the anti-canonical bundle over $\PP^2\times\PP^1$ for genus zero and one. 

For the rest of the section, the specialization
\begin{align*}
    \lambda_{1,k}=e^{\frac{2\pi i k}{3}}\,,\,\lambda_{2,1}=1\,,\,\lambda_{2,2}=-1
\end{align*}
will be fixed.
Since the argument of the proof is parallel to that of Section \ref{MS1}, we mostly omit the proofs whose arguments appeared in Section \ref{MS1}. 

\subsection{Generators.}
From the small $I$-function associated to $((3;2),(0;0))$-twisted theory on $\PP^2\times\PP^1$,
\begin{align*}
    I(q_1,q_2)=\sum_{d_1,d_2\ge 0}q_1^{d_1}q_2^{d_2}
    \frac{\prod_{k=1}^{3d_1+2d_2}(3H_1+2H_2+kz)}{\prod_{i=0}^2\prod_{k=1}^{d_1}(H_1-\lambda_{1,i}+kz)\prod_{j=0}^1\prod_{k=1}^{d_2}(H_1-\lambda_{2,j}+kz)}\,,
\end{align*}
we get the big $\mathds{I}$-function using the argument in \cite[Section 5]{BigI}.

\begin{Prop}
For $\mathsf{t}=t_1H_1+t_2 H_2\in H^*_{\mathsf{T}}([(\CC^3\times\CC^2)/(\CC^*\times\CC^*)],\QQ)$,
\begin{multline}
    \mathds{I}(\mathsf{t})=\sum_{d_1,d_2\ge 0}q_1^{d_1}q_2^{d_2}e^{t_1(H_1+d_1z)/z+t_2(H_2+d_2z)/z}\\
    \cdot\frac{\prod_{k=1}^{3d_1+2d_2}(3H_1+2H_2+kz)}{\prod_{i=0}^2\prod_{k=1}^{d_1}(H_1-\lambda_{1,i}+kz)\prod_{j=0}^1\prod_{k=1}^{d_2}(H_1-\lambda_{2,j}+kz)}\,.
\end{multline}
\end{Prop}

Using Birkhoff factorization (\cite{KL}), an evaluation of the series $\mathds{S}(H_1^i H_2^j)$ can be obtained from $\mathds{I}$-function similar to \eqref{EofS}.

We define the series $\mathds{A}_{i},\mathds{B}_{i},\dots,\mathds{G}_i $ in $q_1, q_2$ by the following equations.

\begin{align*}
    \mathds{S}(\mathds{1})&=\mathds{1}+\frac{\mathds{A}_{1}H_1+\mathds{A}_{2} H_2}{z}+\mathsf{O}(\frac{1}{z})\,,\\
    \mathds{S}(H_1)&=H_1+\frac{\mathds{B}_1 H_1^2+\mathds{B}_2 H_1 H_2+\mathds{B}_3 \mathds{1}}{z}+\mathsf{O}(\frac{1}{z})\,,\\
     \mathds{S}(H_2)&=H_2+\frac{\mathds{C}_1 H_1^2+\mathds{C}_2 H_1 H_2+\mathds{C}_3 \mathds{1}}{z}+\mathsf{O}(\frac{1}{z})\,,\\
     \mathds{S}(H_1^2)&=H_1^2+\frac{\mathds{E}_1\mathds{1}+\mathds{E}_2 H_1^2 H_2+\mathds{E}_3 H_1+\mathds{E}_4 H_2}{z}+\mathsf{O}(\frac{1}{z})\,,\\
      \mathds{S}(H_1 H_2)&=H_1 H_2+\frac{\mathds{F}_1\mathds{1}+\mathds{F}_2 H_1^2 H_2+\mathds{F}_3 H_1+\mathds{F}_4 H_2}{z}+\mathsf{O}(\frac{1}{z})\,,\\
     \mathds{S}(H_1^2 H_2)&=H_1^2 H_2+\frac{\mathds{G}_1 H_1+\mathds{G}_2 H_2+\mathds{G}_3 H_1^3+\mathds{G}_4 H_1 H_2+\mathds{G}_5 \mathds{1}}{z}+\mathsf{O}(\frac{1}{z})\,.
\end{align*}

\subsubsection{Picard-Fuchs equations and asymptotic expansion.}
The function $\mathds{I}$ satisfies the Picard-Fuchs equations.

\begin{align*}
    &\Big(\Big(z\frac{d}{dt_1}\Big)^3-1-q_1\prod_{k=1}^3\Big(3\Big(z\frac{d}{dt_1}\Big)+2\Big(z\frac{d}{dt_2}\Big)+k z\Big)\Big)\mathds{I}=0\,,\\
    &\Big(\Big(z\frac{d}{dt_2}\Big)^2-1-q_2\prod_{k=1}^2\Big(3\Big(z\frac{d}{dt_1}\Big)+2\Big(z\frac{d}{dt_2}\Big)+k z\Big)\Big)\mathds{I}=0\,.
\end{align*}
Denote the small I-function by
$$\overline{\mathds{I}}:=\mathds{I}|_{t_1=0,t_2=0}\,.$$
The restriction $\overline{\mathds{I}}|_{H_1=\lambda_{1,i},H_2=\lambda_{2,j}}$ admits the asymptotic form,

\begin{align*}
    \overline{\mathds{I}}|_{H_1=\lambda_{1,i},H_2}=e^{\frac{\mathds{U}_{ij}}{z}}\Big(\mathds{R}_{0,ij}+\mathds{R}_{1,ij}z+\mathds{R}_{2,ij}z^2+\dots \Big)
\end{align*}
with series $\mathds{U}_{ij},\mathds{R}_{k,ij}\in \CC[[q_1,q_2]]$. Define series $\mathds{L}_{ij}$ and $\mathds{UD}_{ij}$ by
$$\mathds{L}_{ij}=\lambda_{1,i}+q_1\frac{d}{d q_1}\mathds{U}_{ij}\,,\,\,\mathds{UD}_{ij}=\lambda_{2,j}+q_2\frac{d}{dq_2}\mathds{U}_{ij}\,.$$

Let $L_{ij}$ be the series in $q_1$ defined by the constant term with respect to $q_2$. The argument of Lemma \ref{FG1} yields the following lemma.

\begin{Lemma}\label{L321}
We have
$$\mathds{L}_{ij},\mathds{UD}_{ij},\mathds{R}_{n,ij}\in \CC[[q_2]][L_{ij}^{\pm 1}]\,.$$
\end{Lemma}

\subsection{Relations.}
Using the argument in Section \ref{RoG1}, we can find the relations among the series $\mathds{A}_i,\mathds{B}_i,\dots,\mathds{G}_i$. Since this yields complicated expressions, we instead find the relations among the series which are coefficient of $q_2^k$ in $\mathds{A}_i,\mathds{B}_i,\dots$.
Define the series in $q_1$ by the following equations.

\begin{align*}
    \mathds{A}_i(q_1,q_2)=&A_i(q_1)+\sum_{k=1}^{\infty} A_{i,k}(q_1)\,q_2^k\,,\\
    \mathds{B}_i(q_1,q_2)=&B_i(q_1)+\sum_{k=1}^{\infty} B_{i,k}(q_1)\,q_2^k\,,\\
    &\dots\,\\
    \mathds{G}_i(q_1,q_2)=&G_i(q_1)+\sum_{k=1}^{\infty} G_{i,k}(q_1)\,q_2^k\,.
\end{align*}

We get the following results from the argument of Proposition \ref{MG}.
\begin{Prop}\label{L322}
 The series $A_n',B_n',\dots, G_n'$ and $A_{n,k},\,B_{n,k},\,\dots,\,G_{n,k}$ can be represented as rational functions in $B_1',L_{ij}$ for fixed $i\in\{1,2,3\}$ and $j\in \{1,2\}$,
 $$A_n',B_n',\dots\,G_n',A_{n,k},B_{n,k},\dots,G_{n,k}\in\CC(B_1',L_{ij})\,.$$
\end{Prop}

We give the explicit results for the convenience of the reader.
{\tiny
\begin{align*}
    A_1'=&-\frac{1}{(35 + 36 L + 54 L^2) (1 + B_1')^2}(27 - 27 L^3 + 70 B_1' + 35 B_1'^2 + 36 L B_1' (2 + B_1') + 54 L^2 B_1' (2 + B_1'))\,,\\
A_2'=&\frac{1}{3 (35 + 36 L + 54 L^2) (1 + B_1')^2}(54 (-1 + L^3) + 
 2 (-46 + 36 L + 54 L^2 + 81 L^3) B_1' + (-46 + 36 L + 54 L^2 \\&+ 
    81 L^3) B_1'^2)\\
B_2'=&-\frac{1}{3 (35 + 36 L + 54 L^2) (1 + B_1')^2}(108 + 302 B_1' + 256 B_1'^2 + 70 B_1'^3 + 36 L B_1' (4 + 5 B_1' + 2 B_1'^2) + \\&
 54 L^2 B_1' (4 + 5 B_1' + 2 B_1'^2) - 27 L^3 (4 + 6 B_1' + 3 B_1'^2))\\
B_3'=&-\frac{1}{(9 (35 + 36 L + 54 L^2) (1 + B_1')^2)}((36 L B_1'^2 (6 + 8 B_1' + 3 B_1'^2) + 54 L^2 B_1'^2 (6 + 8 B_1' + 3 B_1'^2)\\& - 
  27 L^3 (8 + 36 B_1' + 54 B_1'^2 + 36 B_1'^3 + 9 B_1'^4) + 
  4 (54 + 243 B_1' + 417 B_1'^2 + 313 B_1'^3 + 87 B_1'^4)))\\
C_1'=&0\\
C_2'=&-\frac{1}{(35 + 36 L + 54 L^2) (1 + B_1')^2}(27 - 27 L^3 + 70 B_1' + 35 B_1'^2 + 36 L B_1' (2 + B_1') + 54 L^2 B_1' (2 + B_1'))\\
C_3'=&\frac{1}{3 (35 + 36 L + 54 L^2) (1 + B_1')^2}(36 L B_1' (2 + B_1') + 54 L^2 B_1' (2 + B_1') + 27 L^3 (2 + 6 B_1' + 3 B_1'^2) \\&- 
 2 (27 + 46 B_1' + 23 B_1'^2))\\
E_1'=&B_1'\\
E_2'=&\frac{2 B_1'}{3}\\
E_3'=&\frac{1}{9 (35 + 36 L + 54 L^2) (1 + B_1')^2}(108 (-1 + L^3) + 324 (-1 + L^3) B_1' - 
 6 (62 + 36 L + 54 L^2 - 27 L^3) B_1'^2 \\&- 4 (35 + 36 L + 54 L^2) B_1'^3)\\
E_4'=&-\frac{1}{27 (35 + 36 L + 54 L^2) (1 + B_1')^2}(72 L B_1'^3 (4 + 3 B_1') + 108 L^2 B_1'^3 (4 + 3 B_1') \\&- 
 54 L^3 (2 + 6 B_1' + 3 B_1'^2)^2 + 
 8 (27 + 162 B_1' + 324 B_1'^2 + 278 B_1'^3 + 87 B_1'^4))\\
F_1'=&0\\
F_2'=&B_1'\\
F_3'=&-\frac{1}{3 (35 + 36 L + 54 L^2) (1 + B_1')^2}(108 + 302 B_1' + 256 B_1'^2 + 70 B_1'^3 \\&+ 36 L B_1' (4 + 5 B_1' + 2 B_1'^2) + 
 54 L^2 B_1' (4 + 5 B_1' + 2 B_1'^2) - 27 L^3 (4 + 6 B_1' + 3 B_1'^2))\\
F_4'=&-\frac{1}{9 (35 + 36 L + 54 L^2) (1 + B_1')^2}((36 L B_1'^2 (6 + 8 B_1' + 3 B_1'^2) + 54 L^2 B_1'^2 (6 + 8 B_1' + 3 B_1'^2)\\& - 
  27 L^3 (8 + 36 B_1' + 54 B_1'^2 + 36 B_1'^3 + 9 B_1'^4) + 
  4 (54 + 243 B_1' + 417 B_1'^2 + 313 B_1'^3 + 87 B_1'^4)))\\
G_1'=&0\\
G_2'=&B_1'\\
G_3'=&\frac{2 B_1'}{3}\\
G_4'=&\frac{1}{9 (35 + 36 L + 54 L^2) (1 + B_1')^2}(108 (-1 + L^3) + 324 (-1 + L^3) B_1' - 
 6 (62 + 36 L + 54 L^2 - 27 L^3) B_1'^2 \\&- 4 (35 + 36 L + 54 L^2) B_1'^3)\\
G_5'=&-\frac{1}{27 (35 + 36 L + 54 L^2) (1 + B_1')^2}(72 L B_1'^3 (4 + 3 B_1') + 108 L^2 B_1'^3 (4 + 3 B_1') \\&- 
 54 L^3 (2 + 6 B_1' + 3 B_1'^2)^2 + 
 8 (27 + 162 B_1' + 324 B_1'^2 + 278 B_1'^3 + 87 B_1'^4))
\end{align*}
}

\subsection{Poof of Theorem \ref{MT2}}

The theorem follows from the argument in Section \ref{Mainpf} together with Lemma \ref{L321} and Proposition \ref{L322}.

\section{Elliptic fibration : Threefold}\label{MS3}

\subsection{Overview.}
We study the $((3,3),(0,0))$-twisted theory on $\PP^2\times\PP^2$. This theory recover the standard theory of elliptic fibered Calabi-Yau 3-fold X, defined by the general section of the anti-canonical bundle over $\PP^2\times\PP^2$ for genus zero and one. 

For the rest of the section, the specialization
\begin{align*}
    \lambda_{1,k}=e^{\frac{2\pi i k}{3}}\,,\,\lambda_{2,k}=e^{\frac{2\pi i k}{3}}
\end{align*}
will be fixed.
Since the argument of the proof is parallel to that of Section \ref{MS1}, we mostly omit the proofs whose arguments appeared in Section \ref{MS1}. 

\subsection{Generators.}
From the small $I$-function associated to $((3;3),(0;0))$-twisted theory on $\PP^2\times\PP^2$,
\begin{align*}
    I(q_1,q_2)=\sum_{d_1,d_2\ge 0}q_1^{d_1}q_2^{d_2}
    \frac{\prod_{k=1}^{3d_1+2d_2}(3H_1+3H_2+kz)}{\prod_{i=0}^2\prod_{k=1}^{d_1}(H_1-\lambda_{1,i}+kz)\prod_{j=0}^2\prod_{k=1}^{d_2}(H_1-\lambda_{2,j}+kz)}\,,
\end{align*}
we get the big $\mathds{I}$-function using the argument in \cite[Section 5]{BigI}.

\begin{Prop}
For $\mathsf{t}=t_1H_1+t_2 H_2\in H^*_{\mathsf{T}}([(\CC^3\times\CC^3)/(\CC^*\times\CC^*)],\QQ)$,
\begin{multline}
    \mathds{I}(\mathsf{t})=\sum_{d_1,d_2\ge 0}q_1^{d_1}q_2^{d_2}e^{t_1(H_1+d_1z)/z+t_2(H_2+d_2z)/z}\\
    \cdot\frac{\prod_{k=1}^{3d_1+2d_2}(3H_1+3H_2+kz)}{\prod_{i=0}^2\prod_{k=1}^{d_1}(H_1-\lambda_{1,i}+kz)\prod_{j=0}^2\prod_{k=1}^{d_2}(H_1-\lambda_{2,j}+kz)}\,.
\end{multline}
\end{Prop}

Using Birkhoff factorization (\cite{KL}), an evaluation of the series $\mathds{S}(H_1^i H_2^j)$ can be obtained from $\mathds{I}$-function similar to \eqref{EofS}.

We define the series $\mathds{A}_{i},\mathds{B}_{i},\dots,\mathds{J}_i $ in $q_1, q_2$ by the following equations.

\begin{align*}
    \mathds{S}(\mathds{1})&=\mathds{1}+\frac{\mathds{A}_{1}H_1+\mathds{A}_{2} H_2}{z}+\mathsf{O}(\frac{1}{z^2})\,,\\
    \mathds{S}(H_1)&=H_1+\frac{\mathds{B}_1 H_1^2+\mathds{B}_2 H_1 H_2+\mathds{B}_3 H_2^2}{z}+\mathsf{O}(\frac{1}{z^2})\,,\\
     \mathds{S}(H_2)&=H_2+\frac{\mathds{C}_1 H_1^2+\mathds{C}_2 H_1 H_2+\mathds{C}_3 H_2^2}{z}+\mathsf{O}(\frac{1}{z^2})\,,\\
     \mathds{S}(H_1^2)&=H_1^2+\frac{\mathds{E}_1\mathds{1}+\mathds{E}_2 H_1^2 H_2+\mathds{E}_3 H_1 H_2^2}{z}+\mathsf{O}(\frac{1}{z^2})\,,\\
      \mathds{S}(H_1 H_2)&=H_1 H_2+\frac{\mathds{F}_1\mathds{1}+\mathds{F}_2 H_1^2 H_2+\mathds{F}_3 H_1 H_2^2}{z}+\mathsf{O}(\frac{1}{z^2})\,,\\
     \mathds{S}(H_2^2)&= H_2^2+\frac{\mathds{G}_1\mathds{1}+\mathds{G}_2 H_1^2 H_2+\mathds{G}_3 H_1 H_2^2}{z}+\mathsf{O}(\frac{1}{z^2})\,,\\
     \mathds{S}(H_1^2H_2)&= H_1^2H_2+\frac{\mathds{H}_1H_1+\mathds{H}_2 H_2+\mathds{H}_3 H_1^2H_2^2 }{z}+\mathsf{O}(\frac{1}{z^2})\,,\\
     \mathds{S}(H_1H_2^2)&= H_1H_2^2+\frac{\mathds{I}_1H_1+\mathds{I}_2 H_2+\mathds{I}_3 H_1^2H_2^2 }{z}+\mathsf{O}(\frac{1}{z^2})\,,\\
     \mathds{S}(H_1^2H_2^2)&= H_1^2H_2^2+\frac{\mathds{J}_1H_1^2+\mathds{J}_2 H_1H_2+\mathds{J}_3 H_2^2 }{z}+\mathsf{O}(\frac{1}{z^2})\,.
\end{align*}

\subsubsection{Picard-Fuchs equations and asymptotic expansion.}
The function $\mathds{I}$ satisfies the Picard-Fuchs equations.

\begin{align*}
    &\Big(\Big(z\frac{d}{dt_1}\Big)^3-1-q_1\prod_{k=1}^3\Big(3\Big(z\frac{d}{dt_1}\Big)+3\Big(z\frac{d}{dt_2}\Big)+k z\Big)\Big)\mathds{I}=0\,,\\
    &\Big(\Big(z\frac{d}{dt_2}\Big)^3-1-q_2\prod_{k=1}^3\Big(3\Big(z\frac{d}{dt_1}\Big)+3\Big(z\frac{d}{dt_2}\Big)+k z\Big)\Big)\mathds{I}=0\,.
\end{align*}
Denote the small I-function by
$$\overline{\mathds{I}}:=\mathds{I}|_{t_1=0,t_2=0}\,.$$
The restriction $\overline{\mathds{I}}|_{H_1=\lambda_{1,i},H_2=\lambda_{2,j}}$ admits the asymptotic form,

\begin{align*}
    \overline{\mathds{I}}|_{H_1=\lambda_{1,i},H_2}=e^{\frac{\mathds{U}_{ij}}{z}}\Big(\mathds{R}_{0,ij}+\mathds{R}_{1,ij}z+\mathds{R}_{2,ij}z^2+\dots \Big)
\end{align*}
with series $\mathds{U}_{ij},\mathds{R}_{k,ij}\in \CC[[q_1,q_2]]$. Define series $\mathds{L}_{ij}$ and $\mathds{UD}_{ij}$ by
$$\mathds{L}_{ij}=\lambda_{1,i}+q_1\frac{d}{d q_1}\mathds{U}_{ij}\,,\,\,\mathds{UD}_{ij}=\lambda_{2,j}+q_2\frac{d}{dq_2}\mathds{U}_{ij}\,.$$

Let $L_{ij}$ be the series in $q_1$ defined by the constant term with respect to $q_2$. The argument of Lemma \ref{FG1} yields the following lemma.

\begin{Lemma}\label{L331}
We have
$$\mathds{L}_{ij},\mathds{UD}_{ij},\mathds{R}_{n,ij}\in \CC[[q_2]][L_{ij}^{\pm 1}]\,.$$
\end{Lemma}

\subsection{Relations.}\label{GE1}
Using the argument in Section \ref{RoG1}, we can find the relations among the series $\mathds{A}_i,\mathds{B}_i,\dots,\mathds{J}_i$. Since this yields complicated expressions, we instead find the relations among the series which are coefficient of $q_2^k$ in $\mathds{A}_i,\mathds{B}_i,\dots,\mathds{J}_i$.
Define the series in $q_1$ 
\begin{align*}
    \mathds{A}_i(q_1,q_2)=&A_i(q_1)+\sum_{k=1}^{\infty} A_{i,k}(q_1)\,q_2^k\,,\\
    \mathds{B}_i(q_1,q_2)=&B_i(q_1)+\sum_{k=1}^{\infty} B_{i,k}(q_1)\,q_2^k\,,\\
    &\dots\,\\
    \mathds{J}_i(q_1,q_2)=&J_i(q_1)+\sum_{k=1}^{\infty} J_{i,k}(q_1)\,q_2^k\,.
\end{align*}

We get the following results from the argument of Proposition \ref{MG}.
\begin{Prop}\label{L332}
 The series $A_n',B_n',\dots,J_n'$ and $A_{n,k},B_{n,k},\dots,J_{n,k}$ can be represented as rational functions in $B_1',L_{ij}$ for fixed $i\,,j\in\{1,2,3\}$,
 $$A_n',B_n',\dots,J_n',A_{n,k},B_{n,k}\dots,J_{n,k}\in\CC(B_1',L_{ij})\,.$$
\end{Prop}

We give the explicit results in Appendix \ref{APE} for the convenience of the reader.

\subsection{Poof of Theorem \ref{MT3}}

The theorem follows from the argument in Section \ref{Mainpf} together with Lemma \ref{L331} and Proposition \ref{L332}.

\
\

\section{K3 fibration}\label{MS4}

\subsection{Overview.}
We study the $((4,2),(0,0))$-twisted theory on $\PP^3\times\PP^1$. This theory recover the standard theory of K3-fibered Calabi-Yau 3-fold X, defined by the general section of the anti-canonical bundle over $\PP^3\times\PP^1$ for genus zero and one. 

For the rest of the section, the specialization
\begin{align*}
    \lambda_{1,k}=\sqrt{-1}^k\,,\,\lambda_{2,1}=1\,,\,\lambda_{2,2}=-1
\end{align*}
will be fixed.
Since the argument of the proof is parallel to that of Section \ref{MS1}, we mostly omit the proofs whose arguments appeared in Section \ref{MS1}. 

\subsection{Generators.}\label{K3Gen}
From the small $I$-function associated to $((4;2),(0;0))$-twisted theory on $\PP^3\times\PP^1$,
\begin{align*}
    I(q_1,q_2)=\sum_{d_1,d_2\ge 0}q_1^{d_1}q_2^{d_2}
    \frac{\prod_{k=1}^{4d_1+2d_2}(4H_1+2H_2+kz)}{\prod_{i=0}^3\prod_{k=1}^{d_1}(H_1-\lambda_{1,i}+kz)\prod_{j=0}^1\prod_{k=1}^{d_2}(H_1-\lambda_{2,j}+kz)}\,,
\end{align*}
we get the big $\mathds{I}$-function using the argument in \cite[Section 5]{BigI}.

\begin{Prop}
For $\mathsf{t}=t_1H_1+t_2 H_2\in H^*_{\mathsf{T}}([(\CC^3\times\CC^2)/(\CC^*\times\CC^*)],\QQ)$,
\begin{multline}
    \mathds{I}(\mathsf{t})=\sum_{d_1,d_2\ge 0}q_1^{d_1}q_2^{d_2}e^{t_1(H_1+d_1z)/z+t_2(H_2+d_2z)/z}\\
    \cdot\frac{\prod_{k=1}^{4d_1+2d_2}(4H_1+2H_2+kz)}{\prod_{i=0}^3\prod_{k=1}^{d_1}(H_1-\lambda_{1,i}+kz)\prod_{j=0}^1\prod_{k=1}^{d_2}(H_1-\lambda_{2,j}+kz)}\,.
\end{multline}
\end{Prop}

Using Birkhoff factorization (\cite{KL}), an evaluation of the series $\mathds{S}(H_1^i H_2^j)$ can be obtained from $\mathds{I}$-function similar to \eqref{EofS}.

We define the series $\mathds{A}_{i},\mathds{B}_{i},\dots,\mathds{I}_i $ in $q_1, q_2$ by the following equations.

\begin{align*}
    \mathds{S}(\mathds{1})&=\mathds{1}+\frac{\mathds{A}_{1}H_1+\mathds{A}_{2} H_2}{z}+\mathsf{O}(\frac{1}{z^2})\,,\\
    \mathds{S}(H_1)&=H_1+\frac{\mathds{B}_1 H_1^2+\mathds{B}_2 H_1 H_2+\mathds{B}_3 \mathds{1}}{z}+\mathsf{O}(\frac{1}{z^2})\,,\\
     \mathds{S}(H_2)&=H_2+\frac{\mathds{C}_1 H_1^2+\mathds{C}_2 H_1 H_2+\mathds{C}_3 \mathds{1}}{z}+\mathsf{O}(\frac{1}{z^2})\,,\\
     \mathds{S}(H_1^2)&=H_1^2+\frac{\mathds{E}_1H_1^3+\mathds{E}_2 H_1^2 H_2+\mathds{E}_3 H_1+\mathds{E}_4 H_2}{z}+\mathsf{O}(\frac{1}{z^2})\,,\\
      \mathds{S}(H_1 H_2)&=H_1 H_2+\frac{\mathds{F}_1H_1^3+\mathds{F}_2 H_1^2 H_2+\mathds{F}_3 H_1+\mathds{F}_4 H_2}{z}+\mathsf{O}(\frac{1}{z^2})\,,\\
      \mathds{S}(H_1^3)&=H_1^3+\frac{\mathds{G}_1 \mathds{1}+\mathds{G}_2 H_1^3 H_2+\mathds{G}_3 H_1^2+\mathds{G}_4 H_1 H_2}{z}+\mathsf{O}(\frac{1}{z^2})\,,\\
     \mathds{S}(H_1^2 H_2)&=H_1^2 H_2+\frac{\mathds{H}_1 \mathds{1}+\mathds{H}_2 H_1^3 H_2+\mathds{H}_3 H_1^2+\mathds{H}_4 H_1 H_2}{z}+\mathsf{O}(\frac{1}{z^2})\,,\\
     \mathds{S}(H_1^3 H_2)&=H_1^3 H_2+\frac{\mathds{I}_1 H_1+\mathds{I}_2  H_2+\mathds{I}_3 H_1^3+\mathds{I}_4 H_1^2 H_2}{z}+\mathsf{O}(\frac{1}{z^2})\,.
\end{align*}

\subsubsection{Picard-Fuchs equations and asymptotic expansion.}
The function $\mathds{I}$ satisfies the Picard-Fuchs equations.

\begin{align*}
    &\Big(\Big(z\frac{d}{dt_1}\Big)^4-1-q_1\prod_{k=1}^4\Big(4\Big(z\frac{d}{dt_1}\Big)+2\Big(z\frac{d}{dt_2}\Big)+k z\Big)\Big)\mathds{I}=0\,,\\
    &\Big(\Big(z\frac{d}{dt_2}\Big)^2-1-q_2\prod_{k=1}^2\Big(4\Big(z\frac{d}{dt_1}\Big)+2\Big(z\frac{d}{dt_2}\Big)+k z\Big)\Big)\mathds{I}=0\,.
\end{align*}
Denote the small I-function by
$$\overline{\mathds{I}}:=\mathds{I}|_{t_1=0,t_2=0}\,.$$
The restriction $\overline{\mathds{I}}|_{H_1=\lambda_{1,i},H_2=\lambda_{2,j}}$ admits the asymptotic form,

\begin{align*}
    \overline{\mathds{I}}|_{H_1=\lambda_{1,i},H_2}=e^{\frac{\mathds{U}_{ij}}{z}}\Big(\mathds{R}_{0,ij}+\mathds{R}_{1,ij}z+\mathds{R}_{2,ij}z^2+\dots \Big)
\end{align*}
with series $\mathds{U}_{ij},\mathds{R}_{k,ij}\in \CC[[q_1,q_2]]$. Define series $\mathds{L}_{ij}$ and $\mathds{UD}_{ij}$ by
$$\mathds{L}_{ij}=\lambda_{1,i}+q_1\frac{d}{d q_1}\mathds{U}_{ij}\,,\,\,\mathds{UD}_{ij}=\lambda_{2,j}+q_2\frac{d}{dq_2}\mathds{U}_{ij}\,.$$

Let $L_{ij}$ be the series in $q_1$ defined by the constant term with respect to $q_2$. The argument of Lemma \ref{FG1} yields the following lemma.

\begin{Lemma}\label{L421}
We have
$$\mathds{L}_{ij},\mathds{UD}_{ij},\mathds{R}_{n,ij}\in \CC[[q_2]][L_{ij}^{\pm 1}]\,.$$
\end{Lemma}

\subsection{Relations.}\label{GE2}
Using the argument in Section \ref{RoG1}, we can find the relations among the series $\mathds{A}_i,\mathds{B}_i,\dots,\mathds{I}_i$. Since this yields complicated expressions, we instead find the relations among the series which are coefficient of $q_2^k$ in $\mathds{A}_i,\mathds{B}_i,\dots$.
Define the series in $q_1$ 
\begin{align*}
    \mathds{A}_i(q_1,q_2)=&A_i(q_1)+\sum_{k=1}^{\infty} A_{i,k}(q_1)\,q_2^k\,,\\
    \mathds{B}_i(q_1,q_2)=&B_i(q_1)+\sum_{k=1}^{\infty} B_{i,k}(q_1)\,q_2^k\,,\\
    &\dots\,\\
    \mathds{I}_i(q_1,q_2)=&I_i(q_1)+\sum_{k=1}^{\infty} I_{i,k}(q_1)\,q_2^k\,.
\end{align*}

We get the following results from the argument of Proposition \ref{MG}.

\begin{Prop}\label{L422}
 The series $A_k',B_k',\dots,I_k'$ and $A_{n,k},B_{n,k}\dots,I_{n,k}$ can be represented as rational functions in $A_1',B_2,E_1',L_{ij}$ for fixed $i\in\{1,2,3,4\}$ and $j\in \{1,2\}$,
 \begin{align*}A_n',B_n',\dots,I_n',A_{n,k},B_{n,k},\dots,I_{n,k}\in\CC(A_1',B_2', E_1',L_{ij})\,.
 \end{align*}
\end{Prop}

In the proof of Proposition \ref{L422}, we can show that
the series $A_1',\,E_1',\,L_{ij}$ satisfy the following relations.
\begin{multline*}
    2 E_1' + E_1'^2 + 2 A_1' (1 + E_1')^2 + A_1'^2 (1 + E_1')^2\\ - \frac{
 16 (-1 + L_{ij}^4)}{16 + 1 + (-1)^j 8  L_{ij} + 24  L_{ij}^2 + (-1)^j 32 L_{ij}^3}=0\,.
\end{multline*}

\subsection{Poof of Theorem \ref{MT4}}

The theorem follows from the argument in Section \ref{Mainpf} together with Lemma \ref{L421} and Proposition \ref{L422}.

\appendix

\section{Relations on the generators}

\subsection{Elliptic fibration.}\label{APE}
Recall the series $A_k',B_k'\,\dots\,J_k'$ defined in Section \ref{GE1}. They satisfy the following equations. 
{\tiny
\begin{align*}
    A_1'=&\frac{1}{(1 + B_1')^2 (2 + 3 L + 3 L^2)}(-1 + L^3 - 2 B_1' (2 + 3 L + 3 L^2) - 
 B_1'^2 (2 + 3 L + 3 L^2))\,,\\
 A_2'=&\frac{1}{2 (1 + B_1')^2 (2 + 3 L + 3 L^2)}-2 + 2 L^3 + B_1'^2 (-1 + 3 L + 3 L^2 + 3 L^3) + 
 B_1' (-2 + 6 L + 6 L^2 + 6 L^3)\,,\\
  B_2'=&\frac{1}{2 (1 + B_1')^2 (2 + 3 L + 3 L^2)}(-4 + 4 L^3 - 2 B_1'^3 (2 + 3 L + 3 L^2)+ 
 B_1'^2 (-13 - 15 L - 15 L^2 + 3 L^3) \\ &+ 
 2 B_1' (-7 - 6 L - 6 L^2 + 3 L^3))\,,\\
 B_3'=&\frac{1}{4 (1 + B_1')^2 (2 + 3 L + 3 L^2)}(-8 + 8 L^3 + 36 B_1' (-1 + L^3) + 
 3 B_1'^4 (-5 - 3 L - 3 L^2 + 3 L^3) \\&+ 
 4 B_1'^3 (-13 - 6 L - 6 L^2 + 9 L^3) + 
 6 B_1'^2 (-11 - 3 L - 3 L^2 + 9 L^3))\,,\\
 C_1'=&0\,,\\
 C_2'=&\frac{1}{(1 + B_1')^2 (2 + 3 L + 3 L^2)}(-1 + L^3 - B_1'^2 (2 + 3 L + 3 L^2) - B_1' (4 + 6 L + 6 L^2))\,,\\
 C_3'=&\frac{1}{2 (1 + B_1')^2 (2 + 3 L + 3 L^2)}(-2 + 2 L^3 + B_1'^2 (-1 + 3 L + 3 L^2 + 3 L^3) + 
 B_1' (-2 + 6 L + 6 L^2 + 6 L^3))\,,\\
  \end{align*}}
 
 {\tiny
 \begin{align*}
 E_1'=&\frac{1}{4 (1 + B_1')^2 (2 + 3 L + 3 L^2)}(-4 + 4 L^3 + 36 B_1'^3 (-1 + L^3) + 
 3 B_1'^4 (-5 - 3 L - 3 L^2 + 3 L^3) \\&+ 
 4 B_1' (-4 + 3 L + 3 L^2 + 6 L^3) + 
 8 B_1'^2 (-4 + 3 L + 3 L^2 + 6 L^3))\,,\\
 E_2'=&B_1'\,,\\
 E_3'=&\frac{1}{2 (1 + B_1')^2 (2 + 3 L + 3 L^2)}(-2 + 2 L^3 - 2 B_1'^3 (2 + 3 L + 3 L^2) + 6 B_1' (-1 + L^3) + 
 3 B_1'^2 (-3 - 3 L - 3 L^2 + L^3))\,,\\
 F_1'=&\frac{1}{4 (1 + B_1')^2 (2 + 3 L + 3 L^2)}(-8 + 8 L^3 + 36 B_1' (-1 + L^3) + 
 3 B_1'^4 (-5 - 3 L - 3 L^2 + 3 L^3) \\&+ 
 4 B_1'^3 (-13 - 6 L - 6 L^2 + 9 L^3) + 
 6 B_1'^2 (-11 - 3 L - 3 L^2 + 9 L^3))\,,\\
 F_2'=&B_1'\,,\\
 F_3'=&\frac{1}{2 (1 + B_1')^2 (2 + 3 L + 3 L^2)}(-4 + 4 L^3 - 2 B_1'^3 (2 + 3 L + 3 L^2) \\&+ 
 B_1'^2 (-13 - 15 L - 15 L^2 + 3 L^3) + 
 2 B_1' (-7 - 6 L - 6 L^2 + 3 L^3))\,,\\
 G_1'=&\frac{1}{2 (1 + B_1')^2 (2 + 3 L + 3 L^2)}(-2 + 2 L^3 + B_1'^2 (-1 + 3 L + 3 L^2 + 3 L^3) + 
 B_1' (-2 + 6 L + 6 L^2 + 6 L^3))\,,\\
 G_2'=&0\,,\\
 G_3'=&\frac{1}{(1 + B_1')^2 (2 + 3 L + 3 L^2)}(-1 + L^3 - B_1'^2 (2 + 3 L + 3 L^2) - B_1' (4 + 6 L + 6 L^2))\,,\\
 H_1'=&\frac{1}{2 (1 + B_1')^2 (2 + 3 L + 3 L^2)}(-2 + 2 L^3 - 2 B_1'^3 (2 + 3 L + 3 L^2) + 6 B_1' (-1 + L^3) + 
 3 B_1'^2 (-3 - 3 L - 3 L^2 + L^3))\,,\\
 H_2'=&\frac{1}{4 (1 + B_1')^2 (2 + 3 L + 3 L^2)}(-4 + 4 L^3 + 36 B_1'^3 (-1 + L^3) + 
 3 B_1'^4 (-5 - 3 L - 3 L^2 + 3 L^3) \\&+ 
 4 B_1' (-4 + 3 L + 3 L^2 + 6 L^3) + 
 8 B_1'^2 (-4 + 3 L + 3 L^2 + 6 L^3))\,,\\
 H_3'=&B_1'\,,\\
 I_1'=&\frac{1}{2 (1 + B_1')^2 (2 + 3 L + 3 L^2)}(-4 + 4 L^3 - 2 B_1'^3 (2 + 3 L + 3 L^2) + 
 B_1'^2 (-13 - 15 L - 15 L^2 + 3 L^3) \\&+ 
 2 B_1' (-7 - 6 L - 6 L^2 + 3 L^3))\,,\\
 I_2'=&\frac{1}{4 (1 + B_1')^2 (2 + 3 L + 3 L^2)}(-8 + 8 L^3 + 36 B_1' (-1 + L^3) + 
 3 B_1'^4 (-5 - 3 L - 3 L^2 + 3 L^3) \\&+ 
 4 B_1'^3 (-13 - 6 L - 6 L^2 + 9 L^3) + 
 6 B_1'^2 (-11 - 3 L - 3 L^2 + 9 L^3))\,,\\
 I_3'=&B_1'\,,\\
 J_1'=&B_1'\,,\\
 J_2'=&\frac{1}{2 (1 + B_1')^2 (2 + 3 L + 3 L^2)}(-2 + 2 L^3 - 2 B_1'^3 (2 + 3 L + 3 L^2) + 6 B_1' (-1 + L^3) + 
 3 B_1'^2 (-3 - 3 L - 3 L^2 + L^3))\,,\\
 J_3'=&\frac{1}{4 (1 + B_1')^2 (2 + 3 L + 3 L^2)}(-4 + 4 L^3 + 36 B_1'^3 (-1 + L^3) + 
 3 B_1'^4 (-5 - 3 L - 3 L^2 + 3 L^3) \\&+ 
 4 B_1' (-4 + 3 L + 3 L^2 + 6 L^3) + 
 8 B_1'^2 (-4 + 3 L + 3 L^2 + 6 L^3))\,.
\end{align*}}
Here $L=L_{0,0}$. Similar equations hold for other $(i,j)\ne (0,0)$.

\pagebreak
\subsection{K3 fibration.}
The series $A_k',B_k'\,\dots\,I_k'$ defined in Section \ref{GE1} satisfy the following equations.

{\tiny\begin{align*}
    A_2'&=\frac{1}{4} (7 A_1' + 4 A_1'^2 - 2 B_2' + 8 E_1' + 16 A_1' E_1' + 
   8 A_1'^2 E_1' + 4 E_1'^2 + 8 A_1' E_1'^2 + 4 A_1'^2 E_1'^2)\,,\\
   B_1'&=A_1'\,,\\
   B_3'&=\frac{1}{32 (1 + A_1')}(16 A_1'^4 (1 + E_1')^4 + 
  8 A_1'^3 (1 + E_1')^2 (9 + 16 E_1' + 8 E_1'^2) \\&+ 
  A_1'^2 (105 + 400 E_1' + 584 E_1'^2 + 384 E_1'^3 + 96 E_1'^4 + 
     16 B_2' (1 + E_1')^2) \\&+ 
  4 A_1' (13 + 60 E_1' + 94 E_1'^2 + 64 E_1'^3 + 16 E_1'^4 + 
     B_2' (5 + 16 E_1' + 8 E_1'^2)) \\&+ 
  4 (-3 B_2'^2 + B_2' (-2 + 8 E_1' + 4 E_1'^2) + 
     2 E_1' (6 + 11 E_1' + 8 E_1'^2 + 2 E_1'^3)))\,,\\
C_1'&=0\,,\\
C_2'&=A_1'\,,\\
C_3'&=\frac{1}{4}(7 A_1' + 4 A_1'^2 - 2 B_2' + 8 E_1' + 16 A_1' E_1' + 8 A_1'^2 E_1' + 
  4 E_1'^2 + 8 A_1' E_1'^2 + 4 A_1'^2 E_1'^2)\,,\\
  E_2'&=\frac{1}{4}(9 A_1' + 4 A_1'^2 - 2 B_2' + 6 E_1' + 16 A_1' E_1' + 8 A_1'^2 E_1' + 
  4 E_1'^2 + 8 A_1' E_1'^2 + 4 A_1'^2 E_1'^2)\,,\\
  E_3'&=\frac{1}{32 (1 + A_1')}(16 A_1'^4 (1 + E_1')^4 + 
  8 A_1'^3 (1 + E_1')^2 (7 + 16 E_1' + 8 E_1'^2) \\&+ 
  A_1'^2 (53 + 304 E_1' + 536 E_1'^2 + 384 E_1'^3 + 96 E_1'^4 + 
     16 B_2' (1 + E_1')^2) + \\&
  4 A_1' (4 + 38 E_1' + 82 E_1'^2 + 64 E_1'^3 + 16 E_1'^4 + 
     B_2' (11 + 16 E_1' + 8 E_1'^2)) \\&+ 
  4 (-3 B_2'^2 + 4 B_2' (1 + E_1')^2 + 
     2 E_1' (3 + 9 E_1' + 8 E_1'^2 + 2 E_1'^3)))\,,\\
E_4'&=\frac{1}{64 (1 + A_1')^2}(8 B_2'^3 + 384 A_1'^5 (1 + E_1')^6 + 64 A_1'^6 (1 + E_1')^6 - 
  32 B_2' E_1' (1 + E_1')^2 (2 + E_1') \\&- 
  4 B_2'^2 (1 + 8 E_1' + 4 E_1'^2) + 
  8 E_1' (5 + 43 E_1' + 104 E_1'^2 + 106 E_1'^3 + 48 E_1'^4 + 
     8 E_1'^5) \\&+ 
  4 A_1'^4 (1 + E_1')^2 (213 + 904 E_1' + 1412 E_1'^2 + 960 E_1'^3 + 
     240 E_1'^4 - 8 B_2' (1 + E_1')^2) \\&- 
  2 A_1'^3 (-425 - 2960 E_1' - 8264 E_1'^2 - 11904 E_1'^3 - 
     9376 E_1'^4 - 3840 E_1'^5 - 640 E_1'^6 + 64 B_2' (1 + E_1')^4) \\&+ 
  A_1'^2 (359 + 3216 E_1' + 10444 E_1'^2 + 16512 E_1'^3 + 
     13728 E_1'^4 + 5760 E_1'^5 + 960 E_1'^6 - 
     16 B_2'^2 (1 + E_1')^2 \\&- 
     2 B_2' (83 + 352 E_1' + 560 E_1'^2 + 384 E_1'^3 + 96 E_1'^4)) - 
  4 A_1' (8 B_2'^2 (1 + E_1')^2 + 
     B_2' (19 + 96 E_1' + 176 E_1'^2 + 128 E_1'^3 + 32 E_1'^4) \\&- 
     2 (5 + 92 E_1' + 399 E_1'^2 + 736 E_1'^3 + 664 E_1'^4 + 
        288 E_1'^5 + 48 E_1'^6)))\,,\\
F_1'&=0\,,\\
F_2'&=A_1'\,,\\
F_3'&=B_2;\,,\\
F_4'&=\frac{1}{32 (1 + A_1')}(16 A_1'^4 (1 + E_1')^4 + 
  8 A_1'^3 (1 + E_1')^2 (9 + 16 E_1' + 8 E_1'^2) \\&+ 
  A_1'^2 (105 + 400 E_1' + 584 E_1'^2 + 384 E_1'^3 + 96 E_1'^4 + 
     16 B_2' (1 + E_1')^2) \\&+ 
  4 A_1' (13 + 60 E_1' + 94 E_1'^2 + 64 E_1'^3 + 16 E_1'^4 + 
     B_2' (5 + 16 E_1' + 8 E_1'^2)) \\&+ 
  4 (-3 B_2'^2 + B_2' (-2 + 8 E_1' + 4 E_1'^2) + 
     2 E_1' (6 + 11 E_1' + 8 E_1'^2 + 2 E_1'^3)))\,,
   \end{align*}}
 
 {\tiny
 \begin{align*}
 G_1'&=\frac{1}{128 (1 + A_1')^2}  (64 A_1'^6 (1 + E_1')^6 + 
  16 A_1'^5 (1 + E_1')^4 (23 + 48 E_1' + 24 E_1'^2) \\&+ 
  A_1'^3 (733 + 5376 E_1' + 15616 E_1'^2 + 23168 E_1'^3 + 
     18592 E_1'^4 + 7680 E_1'^5 + 1280 E_1'^6 \\&- 
     16 B_2' (1 + E_1')^2 (9 + 16 E_1' + 8 E_1'^2)) - 
  4 A_1'^4 (8 B_2' (1 + E_1')^4 - 
     3 (1 + E_1')^2 (65 + 288 E_1' + 464 E_1'^2 + 320 E_1'^3 + 
        80 E_1'^4)) \\&- 
  2 A_1'^2 (-137 - 1444 E_1' - 4790 E_1'^2 - 7936 E_1'^3 - 
     6784 E_1'^4 - 2880 E_1'^5 - 480 E_1'^6 \\&+ 8 B_2'^2 (1 + E_1')^2 + 
     B_2' (105 + 400 E_1' + 584 E_1'^2 + 384 E_1'^3 + 96 E_1'^4)) \\&+ 
  8 (B_2'^3 + B_2'^2 (1 - 4 E_1' - 2 E_1'^2) - 
     2 B_2' E_1' (6 + 11 E_1' + 8 E_1'^2 + 2 E_1'^3) + 
     2 E_1' (9 + 17 E_1' + 48 E_1'^2 + 52 E_1'^3 + 24 E_1'^4 + 
        4 E_1'^5)) \\&- 
  4 A_1' (B_2'^2 (5 + 16 E_1' + 8 E_1'^2) + 
     2 B_2' (13 + 60 E_1' + 94 E_1'^2 + 64 E_1'^3 + 16 E_1'^4) \\&- 
     2 (2 + 102 E_1' + 349 E_1'^2 + 696 E_1'^3 + 654 E_1'^4 + 
        288 E_1'^5 + 48 E_1'^6)))\,,\\
G_2'&=\frac{E_1'}{2}\,,\\
G_3'&=\frac{1}{8} (-2 B_2' + 4 A_1'^2 (1 + E_1')^2 + 2 E_1' (3 + 2 E_1') + 
  A_1' (7 + 16 E_1' + 8 E_1'^2))\,,\\
G_4'&=\frac{1}{64 (1 + A_1')}(16 A_1'^4 (1 + E_1')^4 + 
  8 A_1'^3 (1 + E_1')^2 (7 + 16 E_1' + 8 E_1'^2)\\& + 
  A_1'^2 (61 + 304 E_1' + 536 E_1'^2 + 384 E_1'^3 + 96 E_1'^4 + 
     16 B_2' (1 + E_1')^2) \\&+ 
  4 A_1' (6 + 38 E_1' + 82 E_1'^2 + 64 E_1'^3 + 16 E_1'^4 + 
     B_2' (7 + 16 E_1' + 8 E_1'^2)) \\&+ 
  4 (-3 B_2'^2 + 4 B_2' E_1' (2 + E_1') + 
     2 E_1' (3 + 9 E_1' + 8 E_1'^2 + 2 E_1'^3)))\,,\\  
H_1'&=\frac{1}{64 (1 + A_1')^2}  (8 B_2'^3 + 384 A_1'^5 (1 + E_1')^6 + 64 A_1'^6 (1 + E_1')^6 - 
  32 B_2' E_1' (1 + E_1')^2 (2 + E_1') \\&- 
  4 B_2'^2 (1 + 8 E_1' + 4 E_1'^2) + 
  8 E_1' (5 + 43 E_1' + 104 E_1'^2 + 106 E_1'^3 + 48 E_1'^4 + 
     8 E_1'^5) \\&+ 
  4 A_1'^4 (1 + E_1')^2 (213 + 904 E_1' + 1412 E_1'^2 + 960 E_1'^3 + 
     240 E_1'^4 - 8 B_2' (1 + E_1')^2) \\&- 
  2 A_1'^3 (-425 - 2960 E_1' - 8264 E_1'^2 - 11904 E_1'^3 - 
     9376 E_1'^4 - 3840 E_1'^5 - 640 E_1'^6 + 64 B_2' (1 + E_1')^4) \\&+ 
  A_1'^2 (359 + 3216 E_1' + 10444 E_1'^2 + 16512 E_1'^3 + 
     13728 E_1'^4 + 5760 E_1'^5 + 960 E_1'^6 - 
     16 B_2'^2 (1 + E_1')^2 \\&- 
     2 B_2' (83 + 352 E_1' + 560 E_1'^2 + 384 E_1'^3 + 96 E_1'^4)) - 
  4 A_1' (8 B_2'^2 (1 + E_1')^2 \\&+ 
     B_2' (19 + 96 E_1' + 176 E_1'^2 + 128 E_1'^3 + 32 E_1'^4) - 
     2 (5 + 92 E_1' + 399 E_1'^2 + 736 E_1'^3 + 664 E_1'^4 + 
        288 E_1'^5 + 48 E_1'^6)))   \,,\\
H_2'&=E_1'\,,\\
H_3'&=\frac{1}{4}(9 A_1' + 4 A_1'^2 - 2 B_2' + 6 E_1' + 16 A_1' E_1' + 8 A_1'^2 E_1' + 
  4 E_1'^2 + 8 A_1' E_1'^2 + 4 A_1'^2 E_1'^2)\,,\\
H_4'&=\frac{1}{32 (1 + A_1')}(16 A_1'^4 (1 + E_1')^4 + 
  8 A_1'^3 (1 + E_1')^2 (7 + 16 E_1' + 8 E_1'^2) + 
  A_1'^2 (53 + 304 E_1' + 536 E_1'^2 + 384 E_1'^3 + 96 E_1'^4 \\&+ 
     16 B_2' (1 + E_1')^2) + 
  4 A_1' (4 + 38 E_1' + 82 E_1'^2 + 64 E_1'^3 + 16 E_1'^4 + 
     B_2' (11 + 16 E_1' + 8 E_1'^2)) \\&+ 
  4 (-3 B_2'^2 + 4 B_2' (1 + E_1')^2 + 
     2 E_1' (3 + 9 E_1' + 8 E_1'^2 + 2 E_1'^3)))\,,\\
I_1'&=\frac{1}{64 (1 + A_1')}(16 A_1'^4 (1 + E_1')^4 + 
  8 A_1'^3 (1 + E_1')^2 (7 + 16 E_1' + 8 E_1'^2) + 
  A_1'^2 (61 + 304 E_1' + 536 E_1'^2 + 384 E_1'^3 + 96 E_1'^4 \\&+ 
     16 B_2' (1 + E_1')^2) + 
  4 A_1' (6 + 38 E_1' + 82 E_1'^2 + 64 E_1'^3 + 16 E_1'^4 + 
     B_2' (7 + 16 E_1' + 8 E_1'^2)) \\&+ 
  4 (-3 B_2'^2 + 4 B_2' E_1' (2 + E_1') + 
     2 E_1' (3 + 9 E_1' + 8 E_1'^2 + 2 E_1'^3)))\,,\\
        \end{align*}}
 
 {\tiny
 \begin{align*}
I_2'&=\frac{1}{128 (1 + A_1')^2}(64 A_1'^6 (1 + E_1')^6 + 
  16 A_1'^5 (1 + E_1')^4 (23 + 48 E_1' + 24 E_1'^2)\\& + 
  A_1'^3 (733 + 5376 E_1' + 15616 E_1'^2 + 23168 E_1'^3 + 
     18592 E_1'^4 + 7680 E_1'^5 + 1280 E_1'^6 - 
     16 B_2' (1 + E_1')^2 (9 + 16 E_1' + 8 E_1'^2)) \\&- 
  4 A_1'^4 (8 B_2' (1 + E_1')^4 - 
     3 (1 + E_1')^2 (65 + 288 E_1' + 464 E_1'^2 + 320 E_1'^3 + 
        80 E_1'^4)) \\&- 
  2 A_1'^2 (-137 - 1444 E_1' - 4790 E_1'^2 - 7936 E_1'^3 - 
     6784 E_1'^4 - 2880 E_1'^5 - 480 E_1'^6 + 8 B_2'^2 (1 + E_1')^2 \\&+ 
     B_2' (105 + 400 E_1' + 584 E_1'^2 + 384 E_1'^3 + 96 E_1'^4)) + 
  8 (B_2'^3 + B_2'^2 (1 - 4 E_1' - 2 E_1'^2) - 
     2 B_2' E_1' (6 + 11 E_1' + 8 E_1'^2 + 2 E_1'^3)\\& + 
     2 E_1' (9 + 17 E_1' + 48 E_1'^2 + 52 E_1'^3 + 24 E_1'^4 + 
        4 E_1'^5)) - 
  4 A_1' (B_2'^2 (5 + 16 E_1' + 8 E_1'^2)\\& + 
     2 B_2' (13 + 60 E_1' + 94 E_1'^2 + 64 E_1'^3 + 16 E_1'^4) - 
     2 (2 + 102 E_1' + 349 E_1'^2 + 696 E_1'^3 + 654 E_1'^4 + 
        288 E_1'^5 + 48 E_1'^6)))\,,\\
I_3'&=\frac{E_1'}{2}\,,\\
I_4'&=\frac{1}{8}(-2 B_2' + 4 A_1'^2 (1 + E_1')^2 + 2 E_1' (3 + 2 E_1') + 
  A_1' (7 + 16 E_1' + 8 E_1'^2))\,.
\end{align*}}
Here $L=L_{0,0}$. Similar equations hold for other $(i,j)\ne (0,0)$.

\end{document}